\documentclass[11pt]{article}
\usepackage[utf8]{inputenc}

        \usepackage[bottom=4cm, left =2.5cm, right =2.5cm]{geometry} 
        \usepackage{fancyhdr}
        \usepackage{hyperref} 
        \usepackage{amsthm, thmtools,mathtools} 
        \usepackage[intoc, english]{nomencl} 
        
        \usepackage{amsmath, amsfonts, amssymb} 
        \usepackage{stmaryrd} 
        \usepackage{colonequals} 
        \usepackage{bbm} 
        \usepackage{enumitem} 
        
        \allowdisplaybreaks[1]
        \usepackage[english]{babel} 
        \usepackage[T1]{fontenc} 
        \usepackage{authblk} 
        \usepackage[backend=biber,style=alphabetic,sorting=nyt,giveninits=true]{biblatex}
        \usepackage{csquotes}
        \usepackage{xcolor} 
        \definecolor{DarkBlue}{rgb}{0.0, 0.28, 0.39}
        
        \newtheorem{thm}{Theorem}[section]
        \newtheorem{prp}[thm]{Proposition}
        \newtheorem{lem}[thm]{Lemma}
        \newtheorem{cor}[thm]{Corollary}
        
        \newtheorem{df}[thm]{Definition}
        
        \theoremstyle{definition}
        \newtheorem{nota}[thm]{Notations}
        
        \theoremstyle{remark}
        \newtheorem{rem}[thm]{Remark}
        
        \newcommand{\E}{\mathbb{E}}
        \newcommand{\condEp}[2]{\left.\E\left(#1\right\lvert #2 \right)}
        \newcommand{\condEc}[2]{\E\left.\left[#1\right\lvert #2 \right]}
        \newcommand{\R}{\mathbb{R}}
        \newcommand{\Proba}{\mathbb{P}}
        \newcommand{\law}{\mathcal{L}}
        
        \newcommand{\Z}{\mathbb{Z}}
        \newcommand{\calF}{\mathcal{F}}
        
        \newcommand{\indic}{\mathbbm{1}}
        \newcommand{\norm}[1]{\left\Vert #1\right\Vert}
        \newcommand{\normB}[1]{\norm{#1}_{\mathbb B}}
        \newcommand{\abs}[1]{\left\lvert #1 \right\rvert}
        
        \newcommand{\eqdef}{\colonequals}
        \newcommand{\indep}{\perp\!\!\!\perp}
        
        \newcommand\bigp[1]{\left(#1\right)}
        
        \newcommand{\bigcro}[1]{\left[#1\right]}
        
        \newcommand{\eps}{\varepsilon}
        
        \renewcommand{\leq}{\leqslant}
        \renewcommand{\geq}{\geqslant}
        
        \renewcommand{\Tilde}{\widetilde}

        \makeatletter
        \@addtoreset{equation}{section}
        \makeatother
        
        \hypersetup{
            colorlinks=true,       
            linkbordercolor=DarkBlue,
            linkcolor=DarkBlue!70!DarkBlue, 
            citecolor=DarkBlue!70!DarkBlue, 
            urlcolor=DarkBlue!70!DarkBlue 
        }
        
        \AtBeginBibliography{\small}

\title{Rates of convergence in the central limit theorem for Banach valued dependent variables}
\author{Aurélie Bigot\thanks{LAMA, Univ Gustave Eiffel, Univ Paris Est Créteil, UMR 8050 CNRS, F-77454 Marne-La-Vallée, France.}}
\date{}

\begin{document}
\maketitle

\begin{abstract}
    We provide rates of convergence in the central limit theorem in terms of projective criteria for adapted stationary sequences of centered random variables taking values in Banach spaces, with finite moment of order $p \in ]2,3]$ as soon as the central limit theorem holds for the partial sum normalized by $n^{-1/2}$.  
    This result applies to the empirical distribution function in $L^p(\mu)$, where $p\geq 2$ and $\mu$ is a real $\sigma$-finite measure: under some $\tau$-mixing conditions we obtain a rate of order $O(n^{-(p-2)/2})$. In the real case, our result leads to new conditions to reach the optimal rates of convergence in terms of Wasserstein distances of order $p\in ]2,3]$. 
    \\
    
\end{abstract}

\section*{Introduction}
Throughout the paper, $(\mathbb B, \normB{.})$ is a real separable Banach space. Consider $(X_i)_{i \in \Z}$ a stationary sequence of $\mathbb B$-valued centered random variables adapted to a non-decreasing and stationary filtration $(\calF_i)_{i \in \Z}$ and such that $\E \norm{X_0}_{\mathbb B}^2 < \infty$.  For any $n \in \mathbb N^*$ write $S_n = X_1 + \cdots + X_n$. In 2024, the author (see \cite{Big2023}) proved a central limit theorem (in short CLT) for $(n^{-1/2}S_n)_{n \geq 1}$ under the projective condition
    \begin{align}\label{DR_hyp}
        \norm{X_0}_{\mathbb B} \condEp{S_n}{\calF_0}  \text{ converges in } \mathbb{L}^1_{\mathbb B},
    \end{align}
provided that $\mathbb B$ is 2-smooth in the strong sense (see \parencite[(2.2)]{Pin1994}) with a Schauder basis. This result can be viewed as an extension to the Banach space setting of Theorem 1 in \cite{DR2000}. 

Note that other projective criteria leading to a CLT for real-valued r.v.'s have been extended to the Banach spaces setting. Cuny \cite{Cun2017} has extended the Maxwell-Woodroofe theorem (see \cite{MW2000}) to the Banach setting proving that the condition
\begin{align*}
    \sum_{n = 1}^\infty n^{-3/2}(\E[\norm{\condEp{S_n}{\calF_0}}_{\mathbb B}^2])^{1/2} < \infty
\end{align*}
is enough to ensure the CLT for $(n^{-1/2}S_n)_{n \geq 1}$ when the variables take values in a 2-smooth Banach space $\mathbb B$ (see Definition \ref{def:2-smooth}). On another hand, in \cite{DMP2013} the authors extended the Hannan theorem (see \cite{Han1973}) to random variables taking values in a 2-smooth Banach space having a Schauder basis. Hence, they proved that in this context, the condition 
\begin{align*}
    \sum_{n \in \Z} (\E\norm{P_0(X_n)}_{\mathbb B}^2)^{1/2} < \infty
\end{align*}
is enough to ensure the CLT (here $P_0$ is the operator defined by $P_0 = \condEp{\cdot}{\calF_0}-\condEp{\cdot}{\calF_{-1}}$).
It has been proved in \cite{DV2008} that, in the real setting, all these conditions are of independent interest. 

In this paper we are interested in conditions leading to rates of convergence in the CLT when Banach-valued r.v.'s are considered. 
There are different ways to quantify the rate of convergence in the CLT. In this paper, we are interested in quantifying the rate of convergence to zero of 
    \begin{align}
        \label{intro:quantity}
        \Delta_n(f) \eqdef 
        \abs{\E \bigcro{
            f(n^{-1/2}S_n)}
            -\E \bigcro{
            f(G_{})
        }}
    \end{align}
for $f : \mathbb B \to \mathbb R$ belonging to a certain class of functions and where $G$ is a Gaussian $\mathbb B$-valued r.v. whose covariance operator is given in \ref{hyp:CLT} of Theorem \ref{theoreme:vitesse} below. 
In this paper, we shall consider the following class of functions: for $p \geq 1$, let $\Lambda_p(\mathbb B,M)$ be the class of functions $f : \mathbb B \to \R$ $l$-times continuously Fréchet-differentiable such that $\norm{f^{(l)}(0)}\leq M$ and $f^{(l)}$ is $(p-l)$-Hölder continuous with Hölder constant less than or equal to 1 in the sense that 
\begin{align*}
    \norm{f^{(l)}(x) - f^{(l)}(y)} \leq \normB{x-y}^{p-l}, \, \forall x, y \in \mathbb B, 
\end{align*}
where $l$ is the greatest integer strictly less than $p$ and $\norm{.}$ denotes the usual norm on the space of $l$-linear continuous forms.   
In the case of real valued r.v.'s, this class of functions allows to define the so-called Zolotarev distances between probability laws (see our Section \ref{section:reel} for more details). However, in the case of Banach valued r.v.'s, knowing if this class of functions is sufficient to yield the convergence in distribution is not so clear (we refer to Ra\v{c}kauskas and Suquet \cite{RS2023} where the relation between weak convergence of probabilities on a smooth Banach space and uniform convergence over certain classes of smooth functions is established). In Section \ref{section:Lp}, we shall see that this holds in the case where $\mathbb B = L^p(\mu)$, $p\geq 2$ and $\mu$ is a $\sigma$-finite measure. 

Let us now recall some previous results concerning rates in the CLT for Banach-valued r.v.'s in terms of the quantity $\Delta_n(f)$ with $f$ belonging to different classes of functions. In the i.i.d. case, we start with the following result due to \cite{Pau1976} and \cite{Zol1976} and leading to order $O(n^{-1/2})$:
\begin{thm}
    \label{PZ1976:rateCLT}
    Let $(X_i)_{i \in \mathbb Z}$ be a sequence of i.i.d. $\mathbb B$-valued centered random variables such that $\E \normB{X_0}^3 < \infty$. Assume that the CLT applies for $(n^{-1/2}S_n)_{n\geq 1}$ with convergence towards a Gaussian $\mathbb B$-valued random variable $G$. If a functional $f : \mathbb B \to \R$ admits a bounded third Fréchet differential, then $$\E\bigcro{f(n^{-1/2}S_n)} - \E\bigcro{f(G)} = O(n^{-1/2}).$$ 
\end{thm}
This result was extended by Bentkus (\cite{Ben1986}) still in the context of i.i.d. r.v.'s, by providing for $p \in ]2,3]$ an estimate of $\Delta_n(f)$ of order $O(n^{-(p-2)/2})$ if $f(x)$ is bounded by $2^p(2 + \normB{x}^p)$, $f$ is thrice Fréchet-differentiable with some boundary conditions on its differentials and if $\E \normB{X_1}^p < \infty$. 

Now, let $(X_i)_{i \in \Z}$ be a sequence of martingale differences taking values in a Banach space $\mathbb B$ with a Schauder basis, admitting a finite moment of order $2+\delta$ (that is $\E\normB{X_i}^{2+\delta} <\infty$), $\delta \in ]0,1]$, and whose weak conditional moments of order 2 are constant. \parencite[Theorem 3]{BHR1983} proved that
\begin{align*}
    \Delta_n(f) \leq c n^{-(2+\delta)/2}\sum_{i = 1}^n (\E \normB{X_i}^{2+\delta} + \E\normB{G}^{2+\delta})
\end{align*}
for any $f : \mathbb B \to \R$ twice Fréchet-differentiable, whose two first Fréchet-differentials are uniformly continuous and bounded, and such that $f^{(2)}$ is $\delta$-Hölder continuous. Note that this class of functions is slightly more restrictive than $\Lambda_{2+\delta}(\mathbb B, M)$. 

Next, in the context of $\phi$-mixing sequences, Utev (\cite{Ute1991}) established a $\Delta_n(f)$-rate of convergence with a slightly different class of functions than those previously mentioned. His proof is based on blocking techniques and coupling arguments. Let us state \parencite[Theorem 4.1]{Ute1991}{}. Assume that $\mathbb B$ is of type 2 and let $f : \mathbb B \to \R$ be a thrice-Fréchet differentiable function such that there exist $c_1,c_2 > 0$ verifying $\abs{f(x)} \leq c_1 (1 + \normB{x}^3)$ and $\norm{f^{(i)}(x)} \leq c_1(1 + \normB x ^{c_2})$ for $i = 1,2,3$. Then, for any stationary sequence $(X_i)_{i \in \Z}$ of centered $\mathbb B$-valued random variables such that $\E \normB{X_0}^{3+\delta} < \infty$ for some $\delta>0$, 
\begin{align*}
    \Delta_n(f) = O(n^{-1/2})
\end{align*}
provided that the following $\phi$-mixing condition holds: $\phi(k) \leq c k^{-6+10\delta^{-2}}$ (here $(\phi(k))_{k \geq 0}$ is the sequence of $\phi$-mixing coefficients associated with $(X_k)_{k \geq 0}$, see for instance \cite{MPU2019} for the definition). We recall that a Banach space is of type 2 if there exists a constant $c >0$ such that for any sequence $(X_i)_{i \in \Z}$ of $\mathbb B$-valued r.v.'s, independent, centered and with a finite moment of order 2, 
    $\E\bigp{\normB{\sum_{k = 1}^n X_k}^2} \leq c\sum_{k = 1}^n \E\bigp{\normB{X_k}^2}. $

In this paper, we shall be interested in giving sufficient conditions in terms of projective criteria to get rates of convergence for the quantity $\Delta_n(f)$ where $f$ belongs to the class of functions $\Lambda_p(\mathbb B, M)$ for $p \in ]2,3]$ and $M \geq 0$. Our Theorem \ref{theoreme:vitesse} of Section \ref{section:results_B} is in this direction and can be viewed as an extension to dependent sequences of Theorem \ref{PZ1976:rateCLT}. As we shall see in Section \ref{section:reel}, even in the real case our Theorem \ref{theoreme:vitesse} leads to new conditions to reach the rate $O(n^{-\delta/2})$, $0 < \delta \leq 1$, when the r.v.'s have a finite moment of order $2+\delta$. The rest of the paper is organized as follows. We dedicate Section \ref{section:Lp} to the case of Banach spaces $L^p(\mu)$ where $p \geq 2$ and $\mu$ is a real measure. Finally, the proofs of the main results are postponed to Section \ref{section:proofs}. Annex \ref{section:annex} is devoted to the computations of the Fréchet derivatives of some specific functions.   

Throughout the paper, let us denote $\mathbb L^p_{\mathbb B}$ the set of $\mathbb B$-valued random variables $X$ such that $\E \normB{X}^p < \infty$ if $p \in [1, \infty[$ and $\inf\{c > 0 : \normB{X} \leq c \;\, a.s.\} < \infty$ if $p = \infty$.

\section{Rates of convergence in the Banach setting}\label{section:results_B}
\noindent
The main result of this paper is Theorem \ref{theoreme:vitesse} below which provides estimates of $\Delta_n(f)$ in terms of projective conditions as soon as the CLT holds for $(n^{-1/2}S_n)_{n \geq 1}$. 
In this section, $(X_i)_{i \in \Z}$ is a stationary sequence of $\mathbb B$-valued centered random variables in $\mathbb L^2_{\mathbb B}$, adapted to a non-decreasing and stationary filtration $(\mathcal F _i)_{i \in \Z}$. Let us introduce the coefficients that will be used below: 
        \begin{align}
            &\gamma(k) \eqdef \E \bigp{\sup \abs{ \condEp{ A(X_0, X_k)}{\calF_0}}}, 
                \label{def_gamma}\tag{1}
            \\
            &\left. \hspace{-8pt} 
            \begin{array}{l}
                a(k) 
                \eqdef 
                \sup\limits_{i \geq 0} \E \bigp{ \norm{X_{-i}}_{\mathbb B}^\delta \sup\abs{\condEc{A(X_0,X_k)}{\calF_{0}} - \E[A(X_{0},X_k)]}},
                \\
                b(k) 
                \eqdef
                \sup\limits_{j \geq 0} \E \bigp{ \norm{X_0}_{\mathbb B}^\delta \sup\abs{\condEc{A(X_k,X_{k+j})}{\calF_{0}} - \E(A(X_{k},X_{k+j})) } } ,
                \\
                \gamma_{2,\delta}(k) \eqdef \max(a(k), b(k)),
            \end{array}
            \hspace{50pt}\right\}
                \label{def_gamma2}\tag{2}
        \end{align}
    where $\delta \in ]0,1]$ and the suprema are taken over all bilinear continuous forms $A$ such that $\norm A \leq 1$.

\begin{thm}\label{theoreme:vitesse}
    Assume that 
    \begin{enumerate}[label =(\alph*)]
        \item \label{hyp:conv_gamma2} $\sum_{k \geq 1} \gamma(k) < \infty$,
        \item \label{hyp:CLT}  $\bigp{\frac{S_n}{\sqrt n}}_{n\geq 1}$ converges in distribution to $G$ where $G$  is a Gaussian $\mathbb B$-valued random variable whose covariance operator is given by: for any $x^*, y^* \in \mathbb B ^*$, $K_G(x^*, y^*) = \sum_{k \in \Z} cov(x^*(X_0), y^*(X_k))$,
        \item  \label{hyp:u.i.} $\bigp{\frac{\norm{S_n}_{\mathbb B}^2}{n}}_{n\geq 1}$ is uniformly integrable.
    \end{enumerate}
    Then, if $X_0 \in \mathbb L^{2+\delta}_{\mathbb B}$, there exists a positive constant $c_\delta$ such that for any $M \geq 0$,
    \begin{align}
        \label{theoreme:vitesse_bound}
        \sup_{f \in \Lambda_{2+\delta}(\mathbb B, M)}\Delta_n(f)
        &\leq 
            n^{-\delta / 2} \bigp{
                (c_\delta + M)\sum_{k \geq 1} k^{\delta / 2} \gamma(k)
                + 
                \sum_{k = 1}^n (k+2)\gamma_{2,\delta}(k)
                +
                \E \norm{X_0}_{\mathbb B}^{2+\delta} + \E \norm{G}_{\mathbb B}^{2+\delta}
            }
        \\ &\hspace{30pt}
            \eqdef n^{-\delta / 2} b(n, M, \delta).  
                \nonumber
    \end{align}
\end{thm}

\begin{rem}
    It is worth noting that under condition \ref{hyp:conv_gamma2}, the covariance series in condition \ref{hyp:CLT} are absolutely convergent. 
\end{rem}

\begin{rem}
    The constant $c_\delta$ only depends on $\delta$, on the second order moment of $\normB{G}$ and on $\lambda = \sup_{k>0}\E(\normB{S_k}^2)/k$ (which is finite thanks to hypothesis \ref{hyp:u.i.}). 
\end{rem}

\begin{rem}
    In the i.i.d. case, we recover Theorem \ref{PZ1976:rateCLT} when $\delta = 1$. 
\end{rem}

The dependence coefficients defined in (\ref{def_gamma}) and (\ref{def_gamma2}) can be estimated in terms of $\tau$-mixing coefficients as introduced in \cite{DP2005}. Let us recall their definition. 

\begin{df}
    Consider a stationary sequence of random variables $(X_i)_{i \in \Z}$ adapted to a non-decreasing and stationary filtration $(\calF_i)_{i \in \Z}$. We define for any integer $k$, 
    \begin{align*}
        &\tau_1(k) 
        =
        \norm{\sup\left\{ 
            P_{X_k | \calF_0}(f) - P_{X_k}(f) \, : \, f \in \Lambda_1(\mathbb B)
        \right\}}_1
        \\
        \text{and }&
        \\
        &\tau_2(k)
        = 
        \max\bigp{
            \tau_1(k), 
            \sup\limits_{l \geq 0}\tau (\calF_0, (X_k,X_{k+l}))
        },
    \end{align*}
    where $P_{X_k | \calF_0}$ is a regular version of the law of probability of $X_k|\calF_0$, $$\tau(\mathcal M,(X,Y)) = \frac{1}{2} \norm{\sup\left\{ 
            P_{(X,Y) | \mathcal M}(f) - P_{(X,Y)}(f) \, : \, f \in \Lambda_1(\mathbb B\times\mathbb B)
        \right\}}_1,$$
    and $\Lambda_1(E)$ is the space of 1-Lipschitz functions from $E$ to $\R$. On $\mathbb B \times \mathbb B$ we put the following norm: $\norm{(u,v)}_{\mathbb B \times \mathbb B} = \frac{1}{2}(\normB{u} + \normB{v})$. 
\end{df}

\begin{lem}\label{prp:mixing_tau}
    Let $(X_i)_{i \in \Z}$ be a stationary sequence of $\mathbb B$-valued centered random variables adapted to a non-decreasing and stationary filtration $(\mathcal F _i)_{i \in \Z}$ and such that $X_0 \in \mathbb L^{2+\delta}_{\mathbb B}$ for some $\delta \in ]0, 1]$. 
    Then
    \begin{align}
        \max (\gamma(k), \gamma_{2,\delta}(k)) \leq 4 \int_0^{\tau_{{}_2}(k)/2} Q_{\normB{X_0}}^{1+\delta} \circ G_{\normB{X_0}} (u) \, du, 
    \end{align}
    where $Q_{\normB{X_0}}$ is the upper tail quantile function defined by $Q_{\normB{X_0}}(u) = \inf\{t \geq 0 \hspace{-2pt}: \hspace{-2pt}  \Proba(\normB{X_0}>t)\leq u\}$  for any $u \in [0,1]$ and $G_{\normB{X_0}}$ is the inverse of $x \mapsto \int_0^x Q_{\normB{X_0}}(u) du$. 
\end{lem}

\noindent
Now to give sufficient conditions ensuring conditions \ref{hyp:CLT} and \ref{hyp:u.i.} of Theorem \ref{theoreme:vitesse}, we shall consider Banach spaces that are 2-smooth. Let us recall this notion. 

\begin{df}
            \label{def:2-smooth}
        A Banach space $(\mathbb B, \normB{.})$ is said to be 2-smooth if there exists $L \geq 1$ such that for any $x,y \in \mathbb B$, 
        \begin{align*}
            \normB{x+y}^2 + \normB{x-y}^2 \leq 2\normB x ^2 + 2L^2 \normB y ^2 . 
        \end{align*}
        In this case, we say that $(\mathbb B, \normB{.})$ is $(2,L)$-smooth.
    \end{df}
    The notion of 2-smooth Banach spaces is very useful due to the martingale inequality below (see for instance Proposition 1 in \cite{Ass1975}). Assume that $(\mathbb B, \normB{\cdot})$ is (2,$L$)-smooth then for every martingale differences $(d_k)_{1 \leq k \leq n}$, writing $M_n = d_1 + \cdots + d_n$, we have
    \begin{align}
            \label{def:2-smooth_prp}
        \E(\normB{M_n}^2) \leq 2 L^2 \sum_{k = 1}^n \E (\normB{d_k}^2). 
    \end{align}

    As an example, for any $p \geq 2$ and any real measure $\mu$, $L^p(\mu)$ is $(2, \sqrt{p-1})$-smooth (see for instance \parencite[Proposition 2.1]{Pin1994}{}).

Starting from Theorem \ref{theoreme:vitesse} and using Lemma \ref{prp:mixing_tau} we can derive the following result whose proof is postponed to Section \ref{section:proofs}. 
\begin{prp}\label{prp:2smooth+tau}
    Assume that $(\mathbb B, \norm{.}_{\mathbb B})$ is 2-smooth. Let $(X_i)_{i \in \Z}$ be an ergodic stationary sequence of $\mathbb B$-valued centered random variables adapted to a non-decreasing and stationary filtration $(\mathcal F _i)_{i \in \Z}$ and such that $X_0 \in \mathbb L^{2+\delta}_{\mathbb B}$ for some $\delta \in ]0, 1]$. Assume that 
    \begin{align}
            \label{prp:mixing_tau_condition}
        \sum_{k \geq 1}k\int_0^{\tau_{{}_2}(k)/2} Q_{\normB{X_0}}^{1+\delta} \circ G_{\normB{X_0}} (u) \, du < \infty. 
    \end{align}
    Then conditions \ref{hyp:conv_gamma2}, \ref{hyp:CLT} and \ref{hyp:u.i.} of Theorem \ref{theoreme:vitesse} hold and, for any $M \geq 0$ $$\sup_{f \in \Lambda_{2+\delta}(\mathbb B, M)} \Delta_n(f) = O(n^{-\delta/2}).$$ 
\end{prp}

For the reader's convenience, let us give the following result which specifies the rates of decrease of $(\tau_2(k))_{k\geq 1}$ and the moments of $\normB{X_0}$ for (\ref{prp:mixing_tau_condition}) to hold. Its proof follows from \parencite[Proof of Lemma 2]{DD2003}{}.
\begin{cor}
    Assume that $(\mathbb B, \norm{.}_{\mathbb B})$ is 2-smooth. Let $(X_i)_{i \in \Z}$ be an ergodic stationary sequence of $\mathbb B$-valued centered random variables adapted to a non-decreasing and stationary filtration $(\mathcal F _i)_{i \in \Z}$ and such that $X_0 \in \mathbb L^{2+\delta}_{\mathbb B}$ for some $\delta \in ]0, 1]$. Assume that one of the following conditions holds: 
    \begin{enumerate}[label = (\roman*)]
         \item there exists $r \in ]2+\delta, \infty]$ such that $X_0 \in \mathbb L^{r}_{\mathbb B}$ and $\sum_{k > 0} k^{1+(2+2\delta)/(r-2-\delta)}  \tau_{2}(k) < \infty$
         
        \item there exist $r > 2+\delta$ and $c > 0$ such that for any $x > 0$, $\Proba\bigp{ \norm{X_0}_{\mathbb B}^{} > x} \leq \bigp{c/x}^r$ and $\sum_{k > 0} k \tau_{2}(k)^{1-(1+\delta)/(r-1)} < \infty$
        
        \item $\E \bigp{\normB{X_0}^{2+\delta} \ln(1+\normB{X_0})^{2}} < \infty$ and $\tau_2(n) = O(b^n)$ for some $b <1$. 
    \end{enumerate}
    Then $\sum_{k \geq 1}k\int_0^{\tau_{{}_2}(k)} Q_{\normB{X_0}}^{1+\delta} \circ G_{\normB{X_0}} (u) \, du < \infty$ and Proposition \ref{prp:2smooth+tau} applies. 
\end{cor}

For applications in mind, let us give a condition in terms of $\beta$-mixing coefficients that implies (\ref{prp:mixing_tau_condition}). We first recall the definition of such coefficients. 

\begin{df}
    Consider a stationary sequence of $\mathbb B$-valued random variables $(X_i)_{i \in \Z}$ adapted to a non-decreasing and stationary filtration $(\calF_i)_{i \in \Z}$. We define for any  integer $k$, 
    \begin{align*}
        &\beta_2(k)
        = 
        \sup\limits_{l \geq 0}\beta (\calF_0, (X_k,X_{k+l}))
    \end{align*}
    where $\beta(\mathcal M,(X,Y)) = 
        \norm{\sup\left\{ 
            P_{(X,Y) | \mathcal M}(f) - P_{(X,Y)}(f) \, : \, \norm{f}_\infty \leq 1
            \right\}}_1$. 
\end{df}

\begin{lem}
        \label{cor:mixing_beta}
    Let $(X_i)_{i \in \Z}$ be a stationary sequence of $\mathbb B$-valued centered random variables adapted to a non-decreasing and stationary filtration $(\mathcal F _i)_{i \in \Z}$. Let $\delta >0$, assume that 
    \begin{align}
            \label{cor:mixing_beta_condition}
        \sum_{k \geq 1}k\int_0^{\beta_{{}_2}(k)} Q_{\normB{X_0}}^{2+\delta}(u) \, du < \infty. 
    \end{align}
    Then (\ref{prp:mixing_tau_condition}) is satisfied. 
\end{lem}

\begin{proof}
    For any $k$, from \parencite[Lemma 4]{DM2006},
    \begin{align*}
    \tau_{{}_2}(k) 
    \leq 2 \int_0^{\beta_2(k)} Q_{\normB{X_0}}(u) \, du
    \leq 2G_{\normB{X_0}}^{-1}(\beta_2(k)). 
    \end{align*}
    Thus by a change of variables, 
    \begin{align*}
        \int_0^{\tau_{{}_2}(k)/2} Q_{\normB{X_0}}^{1+\delta} \circ G_{\normB{X_0}} (u) \, du
        \leq 
        \int_0^{\beta_2(k)} Q_{\normB{X_0}}^{2+\delta}(v) \, dv. 
    \end{align*}
\end{proof}

\begin{rem}
    Let $(X_i)_{i \in \Z}$ be a stationary sequence of $\mathbb B$-valued centered random variables adapted to a non-decreasing and stationary filtration $(\mathcal F _i)_{i \in \Z}$. Assume that one the following condition holds: 
    \begin{enumerate}[label = (\roman*)]
        \item there exists $r \in ]2+\delta, \infty]$ such that $X_0 \in \mathbb L^{r}_{\mathbb B}$ and $\sum_{k > 0} k^{1+(4+2\delta)/(r-2-\delta)}  \beta_{2}(k) < \infty$

        \item there exist $r > 2+\delta$ and $c > 0$ such that for any $x > 0$, $\Proba\bigp{ \norm{X_0}_{\mathbb B}^{} > x} \leq \bigp{c/x}^r$ and $\sum_{k > 0} k \beta_{2}(k)^{1-(2+\delta)/r} < \infty$

        \item $\E \bigp{\normB{X_0}^{2+\delta} \ln(1+\normB{X_0})^{2}} < \infty$ and $\beta_2(n) = O(b^n)$ for some $b <1$.
    \end{enumerate}
    Then (\ref{cor:mixing_beta_condition}) is satisfied. 
\end{rem}

\section{The case of real-valued random variables}\label{section:reel}

In the real case, Theorem \ref{theoreme:vitesse} gives some rates in the CLT in terms of Zolotarev distances. Let us recall this notion: let $p >0$, for two real probability measures $\mu$ and $\nu$, the Zolotarev distance of order $p$ between $\mu$ and $\nu$ is given by:
\begin{align*}
    \zeta_p(\mu, \nu) = \sup\left\{ \abs{\int_\R f d\mu - \int_\R f d\nu} \, : \, f \in \Lambda_{p}^0(\R) \right\}
\end{align*}
where
    $\Lambda_{p}^0(\R) \eqdef 
    \Lambda_{p}(\R,0) \cap \{f : f^{(i)}(0) = 0, i=0,\cdots,l \}
    ,$
with $l$ the largest integer strictly less than $p$.
One the advantages of this quantity is that it can be compared with the so-called Wasserstein distance as quoted in \cite{Rio2017} and recalled in what follows. Recall that the latter distance is defined for any $p >0$ by 
\begin{align*}
    W_p(\mu, \nu) 
    = \inf\left\{\bigcro{\E \, \abs{X-Y} ^p}^\frac{1}{p} : P_X = \mu, P_Y = \nu\right\}
    = \norm{F_\mu^{-1} - F_\nu^{-1}}_{L^p}
\end{align*}
where $F_\mu^{-1}, F_\nu^{-1}$ are the generalized inverses of the cumulative distribution functions respectively of the real probability measures $\mu$ and $\nu$. Theorem 1 in \parencite[]{Rio1998} states the following comparison between Zolotarev and Wasserstein distances: for any $p \geq 1$ there exists a positive constant $c_p$ such that
\begin{align}\label{comparison_Wasserstein_Zolotarev}
    W_p(\mu, \nu) \leq c_p \zeta_{p}^{1/p}(\mu, \nu). 
\end{align}

In the real case, it is shown in \parencite[]{Basu1980} that for sequences $(X_n)_{n \in \Z}$ of martingale differences with finite moment of order 3, the following estimate holds: for any uniformly continuous function $f$ twice differentiable and such that $f^{(2)}$ is $\delta$-Hölder continuous for some $\delta \in ]0, 1]$, 
\begin{align*}
    \Delta_n(f) = O(n^{-\delta/2})
\end{align*}
where $G$ is a centered Gaussian r.v. with the same variance as $X_0$.

Relaxing the martingale assumption, Dedecker, Merlevède and Rio (\parencite{DMR2009}) have in particular proved that for a stationary sequence of real random variables in $\mathbb L^p$, with $2< p < 3$, adapted to a non-decreasing stationary filtration $(\calF _k)_k$ and satisfying 
\begin{align}
    &\sum_{n \geq 1} \condEp{X_n}{\calF_0} \text{ converges in }\mathbb L^p 
    \label{DMR2009:cond(3.1)}
    \\
    \text{and }\quad &
    \nonumber
    \\
    &\sum_{n = 1}^\infty \frac{1}{n^{2-p/2}} \norm{\condEp{\frac{S_n^2}{n} - \sigma^2}{\calF_0}}_{p/2} < \infty , 
    \,
    \text{where }\sigma^2 = \lim_{n \to +\infty} n^{-1}\E(S_n^2),
    \label{DMR2009:cond(3.2)}
\end{align}
the following estimate holds: for any $r \in [p-2,p]$,
\begin{align*}
    \zeta_r(P_{n^{-1/2}S_n}, P_{G}) = O(n^{1-p/2}) 
\end{align*}
where $G$ is a centered Gaussian r.v. with variance $\sigma^2$. We infer from (\ref{comparison_Wasserstein_Zolotarev}) that
\begin{align*}
    W_r(P_{n^{-1/2}S_n}, P_{G}) = O(n^{-(p-2)/(2\max(1,r))}).
\end{align*}
\\
Since for real-valued r.v.'s the dependence coefficients defined in (\ref{def_gamma}) and (\ref{def_gamma2}) simplify, let us give a precise statement of Theorem \ref{theoreme:vitesse} in this case.
\begin{thm}\label{theoreme:vitesse_reel}
    Let $(X_i)_{i \in \Z}$ be an ergodic stationary sequence of real centered and square-integrable random variables, adapted to a non-decreasing and stationary filtration $(\mathcal F _i)_{i \in \Z}$. 
    Let us consider for some $\delta \in ]0, 1]$,
        \begin{align*}
            &\Tilde{\gamma}(k) \eqdef \E \bigp{\abs{X_0 \condEc{X_k}{\calF_0}}}, 
            \\
            &\Tilde{a}(k) 
            \eqdef 
            \sup\limits_{i \geq 0} \E \bigp{ \abs{X_{-i}}^\delta \abs{\condEc{X_0X_k}{\calF_{0}} - \E[X_{0}X_k]}},
            \\
            &\Tilde{b}(k) 
            \eqdef
            \sup\limits_{j \geq 0} \E \bigp{ \abs{X_0}^\delta \abs{\condEc{X_k X_{k+j}}{\calF_{0}} - \E(X_{k} X_{k+j})} } ,
            \\
            \textrm{and}&
            \\
            &\Tilde{\gamma_{2,\delta}}(k) \eqdef \max(\Tilde{a}(k), \Tilde{b}(k)). 
        \end{align*}
    Assume that
        $\sum_{k \geq 1} \Tilde{\gamma}(k) < \infty$. 
    Then, if $\E \abs{X_0}^{2+\delta}< \infty$, there exists a positive constant $c_\delta$ such that
    \begin{align*}
        &\zeta_{2+\delta}(P_{n^{-1/2}S_n}, P_{G})
        \\ \nonumber
        &\hspace{25pt}
            \leq 
            n^{-\delta / 2} \bigp{
        c_\delta \, \sum_{k \geq 1} k^{\delta / 2} \Tilde{\gamma}(k)
        + 
        \sum_{k = 1}^n (k+2)\Tilde{\gamma_{2,\delta}}(k)
        +
        \E \abs{X_0}^{2+\delta} + \E \abs{G}^{2+\delta}
        },  
    \end{align*}
    where $G$ is a centered Gaussian random variable whose variance is given by $\E G^2 = \sum_{k \in \mathbb Z} cov(X_0, X_k)$.   
\end{thm}

In the real case, the dependence coefficients can be estimated with the help of the $\alpha$-dependent coefficients as introduced in \cite{DP2007} (see also \parencite[Section 5]{MPU2019}{}). For the reader's convenience, let us recall their definition. 
\begin{df}
    Consider a stationary sequence of real-valued r.v.'s $(X_i)_{i \in \Z}$ adapted to a non-decreasing and stationary filtration $(\calF_i)_{i \in \Z}$. For any $k \geq 0$, let
    \begin{align*}
        &\alpha_2(k) = 
            \sup_{l \geq 0} \alpha \bigp{\calF_0, (X_k, X_{k + l})}       
    \end{align*}
    where, denoting $Z^{(0)} = Z - \E (Z)$,  
\begin{align*}
    \alpha(\mathcal M, (X,Y)) = \sup\limits_{x,y \in \R} \norm{
        \condEp{\indic_{X \leq x}^{(0)}\indic_{Y \leq y}^{(0)}}{\mathcal M} - \E\bigp{\indic_{X \leq x}^{(0)}\indic_{Y \leq y}^{(0)}}
        }_1.
\end{align*}

\end{df}

\begin{cor}\label{cor:mixing_realcase}
    Let $(X_i)_{i \in \Z}$ be a stationary sequence of centered real-valued r.v.'s adapted to a non-decreasing and stationary filtration $(\mathcal F _i)_{i \in \Z}$. Assume that $(\Tilde{\gamma_{2,\delta}}(k))_{k \geq 1}$ decreases towards $0$ and $\E \abs{X_0}^{2+\delta}< \infty$ for some $\delta \in ]0, 1]$ and consider the conditions 
    \begin{enumerate}[label = (\roman*)]
        \item $\sum_{k \geq 1}k\int_0^{\alpha_{{}_2}(k)} Q_{\abs{X_0}}^{2+\delta}(u) \, du < \infty$
        \label{cor_cond_suff:item1_reel}
        \item $\sum_{k \geq 1}k\int_0^{\tau_{{}_2}(k)/2} Q_{\abs{X_0}}^{1+\delta} \circ G_{\abs{X_0}} (u) \, du < \infty$.
        \label{cor_cond_suff:item2_reel}
    \end{enumerate}
    If either \textit{(i)} or \textit{(ii)} is satisfied, the CLT applies for $\bigp{n^{-1/2}S_n}_{n\geq 1}$ and
    \begin{align*}
        \zeta_{2+\delta}(P_{n^{-1/2}S_n}, P_{G}) = O(n^{-\delta /2})
    \end{align*}
    where the Gaussian random variable $G$ is as in Theorem \ref{theoreme:vitesse_reel}.
\end{cor}

\begin{rem}
    To prove that under condition \textit{\ref{cor_cond_suff:item1_reel}} the series $\sum_{k \geq 1} k^{\delta / 2} \Tilde\gamma(k)$ and $\sum_{k \geq 1} k\Tilde{\gamma_{2,\delta}}(k)$ are convergent, we can use the arguments developed in \parencite[pages 201-204]{MPU2019}{}. 
\end{rem}

\begin{rem}
        In this theorem, we do not need to assume the sequence to be ergodic. Indeed, the fact that $(\alpha_{2}(k))_k$ decreases towards $0$ implies 2-ergodicity as defined in \parencite[Definition 1.2]{DS2017}{}.
    \end{rem}

Recalling the inequality (\ref{comparison_Wasserstein_Zolotarev}), it follows that under the assumptions of Corollary \ref{cor:mixing_realcase}, 
\begin{align}
    W_{2+\delta}(P_{n^{-1/2}S_n}, P_{G}) = O(n^{-\delta /(4+2\delta)}). 
\end{align}
This bound was obtained by Sakhanenko (\cite{Sak1985})  in the independent case. Furthermore this rate cannot be improved as indicated in \cite{Rio2009}. 

\begin{rem}
    Conditions (\ref{DMR2009:cond(3.1)}) and (\ref{DMR2009:cond(3.2)}) involved in \parencite[Theorem 3.1]{DMR2009}{} are not comparable with those of our Theorem \ref{theoreme:vitesse_reel}. However, proceeding as in \parencite[pages 203-204]{MPU2019}{}, note that (\ref{DMR2009:cond(3.1)}) holds with $p=2+\delta$ provided that
    $
        \sum_{k \geq 1} k^{p-1}\int_0^{\alpha_2(k)} Q_{\abs{X_0}}^p(u) \, du < \infty
    $
    which is a more restrictive condition than those involved in our Corollary \ref{cor:mixing_realcase}. 
\end{rem}

\section{The case of \texorpdfstring{$L^p(\mu)$}{}-valued random variables}\label{section:Lp}

Let $p \geq 2$ and $\mu$ a $\sigma$-finite measure on $\R$. The space $L^p(\mu)$ equipped with its usual norm is $2$-smooth (even in the strong sense, see Section 2 in \cite{Pin1994}). Hence, Proposition \ref{prp:2smooth+tau} can be applied with $\mathbb B = L^p(\mu)$. Moreover, from Theorem 16 in \cite{RS2023}, we infer that for any $p \geq 1$, in $\mathbb B = L^p(\mu)$ the class of functions $\Lambda_p(\mathbb B,1) \cap B_{||.||_\infty}(0,1)$ characterizes the convergence in distribution.
Note that, as $L^p(\mu)$ admits a continuous modulus of convexity, this can also be shown by Theorem 1 in \cite{Bau1999} combined with Proposition 5 in \cite{BF1966}. 

\noindent
Thus in $L^p(\mu)$, Theorem \ref{theoreme:vitesse} quantifies the convergence in distribution in the CLT. 
\\[6pt]

Let us now apply Theorem \ref{theoreme:vitesse} to the specific functions $x \mapsto \norm{x}_{L^p (\mu)}^q$ for some $p \geq 2$ and $q>2$.
\\

First, when $p = 2$, $\norm{.}_{L^2(\mu)}^2 \in \Lambda_{2+\delta}(L^2(\mu), 2)$ for any $0 < \delta \leq 1$. Indeed, it is easy to see that $\norm{.}_{L^2(\mu)}^2$ is infinitely Fréchet-differentiable and its second order Fréchet-differential at point $x \in L^2(\mu)$ is given by $(u,v) \mapsto 2 \langle u,v \rangle_{L^2(\mu)}$. Then if $\E \norm{X_0}_{L^2(\mu)}^{2+\delta} < \infty$ and if \ref{hyp:conv_gamma2}, \ref{hyp:CLT} and \ref{hyp:u.i.} are verified, 
\begin{align}
        \label{Lp:estimation_p2}
    &\abs{\E \norm{n^{-1/2}S_n}_{L^2(\mu)}^2 -  \E \norm{G}_{L^2(\mu)}^2}
    \nonumber
    \\& \hspace{11pt}
    \leq n^{-\delta / 2} \bigp{
        (c_\delta+2) \, \sum_{k \geq 1} k^{\delta / 2} \gamma(k)
        + 
        \sum_{k = 1}^n (k+2)\gamma_{2,\delta}(k)
        +
        \E \norm{X_0}_{L^2(\mu)}^{2+\delta} + \E \norm{G}_{L^2(\mu)}^{2+\delta}}. 
\end{align}

Next, for $p \in ]2,3]$, according to \parencite[Section 4]{BF1966}{}, $\norm{.}_{p}^{p} \in \Lambda_p(L_p(\mu),0)$. Consequently, for any $L^p(\mu)$-valued stationary sequence of random variables with finite moment of order $p$ with $2<p\leq 3$, if \ref{hyp:conv_gamma2}, \ref{hyp:CLT} and \ref{hyp:u.i.} are verified, 
\begin{align}
        \label{Lp:estimation_2p3}
    &\abs{\E \norm{n^{-1/2}S_n}_{L^p(\mu)}^p -  \E \norm{G}_{L^p(\mu)}^p}
    \nonumber
    \\& 
    \leq n^{-(p-2) / 2} \bigp{
        c_{p-2} \, \sum_{k \geq 1} k^{(p-2) / 2} \gamma(k)
        + 
        \sum_{k = 1}^n (k+2)\gamma_{2,p-2}(k)
        +
        \E \norm{X_0}_{L^p(\mu)}^{p} + \E \norm{G}_{L^p(\mu)}^{p}}. 
\end{align}

We now turn to the case $p > 3$. In this case, we need a preliminary result. 
\begin{lem}
    \label{lemma:diff3_normLp}
    For any $p \geq 3$ and any $q >0$, $\psi_p^q : x \mapsto \norm{x}_{L^p(\mu)}^q$ is thrice Fréchet-differentiable on $L^p(\mu)\setminus \{0\}$ and for any $x \in L^p(\mu)\setminus \{0\}$ and any $h_1, h_2, h_3 \in L^p(\mu)$,
        \begin{align}
      {\psi_p^q}^{(2)}(x)(h_1, h_2)
      &= 
      q(p-1)\norm{x}_{L^p(\mu)}^{q -p}\int h_1 h_2 \abs x^{p-2}d\mu 
        \nonumber
      \\
      &\hspace{33pt}
      + 
      q(q-p)\norm{x}_{L^p(\mu)}^{q - 2p}\int h_1 x \abs{x}^{p-2} d\mu \int h_2 x \abs{x}^{p-2} d\mu
        \nonumber
    \end{align}
    and 
    \begin{align}
        \hspace{22pt}&\hspace{-22pt}{\psi_p^q}^{(3)}(x)(h_1,h_2,h_3)
        =
        q(p-1)(p-2)\norm{x}_{L^p(\mu)}^{q-p} \int h_1h_2h_3x\abs{x}^{p-4}d\mu
            \nonumber
        \\
        +&
        q(p-1)(q-p)\norm{x}_{L^p(\mu)}^{q-2p}\left(
            \int h_1 x \abs{x}^{p-2}d\mu \int h_2h_3 \abs{x}^{p-2}d\mu
            \right.
                \nonumber
                \\
            &\hspace{0.2\textwidth}+\left.
            \int h_2 x \abs{x}^{p-2}d\mu \int h_1h_3 \abs{x}^{p-2}d\mu
            +
            \int h_3 x \abs{x}^{p-2}d\mu \int h_1h_2 \abs{x}^{p-2}d\mu
        \right)
            \nonumber
        \\
        +&
        q(q-p)(q-2p)\norm{x}_{L^p(\mu)}^{q-3p}\int h_1 x \abs{x}^{p-2}d\mu \int h_2 x \abs{x}^{p-2}d\mu \int h_3 x \abs{x}^{p-2}d\mu
        .
            \nonumber
    \end{align}
    In addition, if $q >2$ then $\psi_p^q$ is twice-differentiable at $0$ with ${\psi_p^q}^{(1)}(0)(h_1) = 0$ and ${\psi_p^q}^{(2)}(0)(h_1,h_2) = 0$. 
\end{lem}
We postpone the proof to the Annex \ref{section:annex}. As a consequence, we have the following result. 

\begin{prp}
    \label{prp:diff3_bornee}
    For any $p \geq 3$, there exists a constant $c_p >0$ such that $c_p^{-1}\psi_p^3$ is an element of $\Lambda_3(L^p(\mu), 0)$. 
\end{prp}
Consequently, for any $p \geq 3$ and any $L^p(\mu)$-valued stationary sequence of random variables with finite moment of order $3$, if \ref{hyp:conv_gamma2}, \ref{hyp:CLT} and \ref{hyp:u.i.} are satisfied, 
\begin{align}
        \label{Lp:estimation_p3}
    &\abs{\E \norm{n^{-1/2}S_n}_{L^p(\mu)}^3 -  \E \norm{G}_{L^p(\mu)}^3}
    \nonumber
    \\& \hspace{11pt}
    \leq n^{-1 / 2} c_p \bigp{
        c_\delta \, \sum_{k \geq 1} k^{1 / 2} \gamma(k)
        + 
        \sum_{k = 1}^n (k+2)\gamma_{2,1}(k)
        +
        \E \norm{X_0}_{L^p(\mu)}^{3} + \E \norm{G}_{L^p(\mu)}^{3}}. 
\end{align}

\begin{proof}[Proof of Proposition \ref{prp:diff3_bornee}]
    According to Lemma \ref{lemma:diff3_normLp} with $q = 3$, for any $x, h_1, h_2, h_3 \in L^p(\mu)$ with $x \neq 0$,
    \begin{align}
        \hspace{22pt}&\hspace{-22pt}{\psi_p^3}^{(3)}(x)(h_1,h_2,h_3)
        =
        3(p-1)(p-2)\norm{x}_{L^p(\mu)}^{3-p} \int h_1h_2h_3x\abs{x}^{p-4}d\mu
            \nonumber
        \\
        +&
        3(p-1)(3-p)\norm{x}_{L^p(\mu)}^{3-2p}\left(
            \int h_1 x \abs{x}^{p-2}d\mu \int h_2h_3 \abs{x}^{p-2}d\mu
            \right.
                \nonumber
                \\
            &\hspace{0.2\textwidth}+\left.
            \int h_2 x \abs{x}^{p-2}d\mu \int h_1h_3 \abs{x}^{p-2}d\mu
            +
            \int h_3 x \abs{x}^{p-2}d\mu \int h_1h_2 \abs{x}^{p-2}d\mu
        \right)
            \nonumber
        \\
        +&
        3(3-p)(3-2p)\norm{x}_{L^p(\mu)}^{3-3p}\int h_1 x \abs{x}^{p-2}d\mu \int h_2 x \abs{x}^{p-2}d\mu \int h_3 x \abs{x}^{p-2}d\mu. 
            \label{normp^3_diff:diff3}
    \end{align}
    Applying Hölder's inequality, for any $w,y,z \in L^p(\mu)$,
    \begin{align*}
        &\abs{\int wyz x\abs{x}^{p-4}d\mu}
        \leq 
        \norm{w}_{L^p(\mu)} \norm{y}_{L^p(\mu)} \norm{z}_{L^p(\mu)} \norm{x}_{L^p(\mu)}^{p-3}
        \\
        &\abs{\int y x \abs{x}^{p-2}d\mu} 
        \leq 
        \norm{y}_{L^p(\mu)} \norm{x}_{L^p(\mu)}^{p-1}
        \\
        &\abs{\int yz \abs{x}^{p-2}d\mu} 
        \leq 
        \norm{y}_{L^p(\mu)} \norm{z}_{L^p(\mu)} \norm{x}_{L^p(\mu)}^{p-2}.
    \end{align*}
    Using these upper bounds in (\ref{normp^3_diff:diff3}), we get 
    \begin{align*}
        \abs{{\psi_p^3}^{(3)}(x)(h_1,h_2,h_3)}
        \leq 6(p-1)^2 \norm{h_1}_{L^p(\mu)}\norm{h_2}_{L^p(\mu)}\norm{h_3}_{L^p(\mu)}.
    \end{align*}
    Hence ${\psi_p^3}^{(3)}$ is bounded by $6(2p^2-8p+7)$ on $L^p(\mu) \setminus \{0\}$ and for any $x,y \in L^p(\mu) \setminus \{0\}$, 
    \begin{align*}
        \norm{{\psi_p^3}^{(2)}(x) - {\psi_p^3}^{(2)}(y)} \leq 6(2p^2-8p+7) \norm{x-y}_{L^p(\mu)}. 
    \end{align*}
    Furthermore, applying Hölder's inequality, for any $x \in L^p(\mu) \setminus \{0\}$, $\norm{{\psi_p^3}^{(2)}(x) - {\psi_p^3}^{(2)}(0)} \leq 6\norm{x}_{L^p(\mu)}^{}$.
    The result follows with $c_p = 6(2p^2-8p+7)$. 
    \
\end{proof}

\subsection{Application to the empirical distribution function in \texorpdfstring{$L^p(\mu)$}{}}

Let us consider $(Y_i)_{i \in \Z}$ a stationary and ergodic sequence of real random variables, whose cumulative distribution function is denoted $F$, and define $F_n(t) = \frac 1 n \sum_{k = 1}^n \indic_{Y_k \leq t}$ the empirical distribution function. We suppose that
    \begin{align*}
        \int_{\R_-} F(t)^p \, d\mu(t) + \int_{\R_+} (1-F(t))^p \, d\mu(t) < \infty
    \end{align*}
so that $F_n - F$ is an element of $L^p(\mu)$, with $\mu$ a  $\sigma$-finite real measure on $\R$ and $p \geq 2$. Define the random process: 
    \begin{align*}
        \forall k \in \Z, \, X_k = \{ \indic_{Y_k \leq t} - F(t)\, : \, t \in \R\}
    \end{align*}
which takes values in $L^p(\mu)$. With such a notation, the study of the asymptotic behavior of $\bigp{\sqrt n (F_n - F)}_{n \geq 1}$ is equivalent to the study of the asymptotic behavior of $(n^{-1/2}S_n)_{n \geq 1}$ in $L^p(\mu)$ where $S_n = X_1 + \cdots + X_n$.
The CLT for empirical distribution functions in $L^p(\mu)$ has already been studied in some papers, we can mention for instance \cite{Dede2009}, \cite{DM2007}, \cite{DM2017}, \cite{Cun2017} or \cite{Big2023}.

    In \cite{DM2007} the link between the convergence in distribution of $\sqrt n (F_n - F)$ in $L ^p(\mu)$ and Donsker classes has been clearly established. More precisely, let us denote 
    \begin{align}
        \label{Lp:def_Sobolev_W1q}
        W_{1,q}(\mu) \eqdef \left\{f : \R \rightarrow \R \, : \, f(x) = f(0) + \indic_{x >0} \int_{[0,x[}g \, d\mu \, - \indic_{x\leq 0} \int_{[x, 0[} g \, d\mu \, , \, \norm{g}_{L^q(\mu)} \leq 1 \right\}
        ,
    \end{align}
    where $q$ is the conjugate exponent of $p$. Then according to \parencite[Lemma 1]{DM2007} the following convergences are equivalent: 
    \begin{enumerate}[label = (\roman*)]
            \item $
            \{\sqrt n (F_n - F)(t) \}_{_t} \xrightarrow[]{\law} \{ G(t)\}_{_t} \, \textrm{in } L^p(\mu)
            $
    
            \item $
            \left\{ \sqrt n \bigp{\frac 1 n \sum_{i= 1}^n f(Y_i) - \E f(Y_0)} \right\}
            \xrightarrow[]{\law} \{G_1(f)\} \, \textrm{in } \ell^\infty (W_{1,q}(\mu))
            $   
    \end{enumerate}
    where $\ell^\infty (W_{1,q}(\mu))$ is the space of all functions $\phi : W_{1,q}(\mu) \to \R$ such that $\sup_{f \in W_{1,q}(\mu)} \abs{\phi(f)}$ is finite, and $G_1(f) = \int g(t)G(t)d\mu(t)$ where $g$ is defined in (\ref{Lp:def_Sobolev_W1q}). 
    Hence, proving that $W_{1,q}(\mu)$ is a Donsker class for $(Y_n)_{n \geq 1}$ is equivalent to proving that $\bigp{n^{-1/2}S_n}_{n \geq 1}$ satisfies the CLT. Rewriting (\ref{Lp:estimation_p2}), (\ref{Lp:estimation_2p3}) and (\ref{Lp:estimation_p3}) in terms of Donsker classes, it follows that if \ref{hyp:conv_gamma2}, \ref{hyp:CLT} and \ref{hyp:u.i.} in Theorem \ref{theoreme:vitesse} are verified, 
    \begin{itemize}
        \item if $p = 2$ and $\E \norm{X_0}_{L^2(\mu)}^{2+\delta} < \infty$ with some $0 < \delta \leq 1$, 
            \begin{align*}
                \abs{\E \bigp{\sup\left\{
                    \sqrt{n}\abs{\frac 1 n \sum_{i = 1}^n f(Y_i) - \mu(f)} 
                    \, : \, 
                    f \in W_{1, 2}(\mu)
                \right\}^2}
                - \E \norm{G}_{L^2(\mu)}^2
                }
                \leq n^{-\delta/2} b(n,2, \delta),
            \end{align*}

        \item if $2 < p < 3$ and $\E \norm{X_0}_{L^p(\mu)}^{p} < \infty$, 
            \begin{align*}
                \abs{\E \bigp{\sup\left\{
                    \sqrt{n}\abs{\frac 1 n \sum_{i = 1}^n f(Y_i) - \mu(f)} 
                    \, : \, 
                    f \in W_{1, q}(\mu)
                \right\}^{p}}
                - \E \norm{G}_{L^p(\mu)}^{p}
                }
                \leq n^{-(p-2)/2} b(n,0, p-2),
            \end{align*}

        \item if $p \geq 3$ and $\E \norm{X_0}_{L^p(\mu)}^{3} < \infty$, 
            \begin{align*}
                \abs{\E \bigp{\sup\left\{
                    \sqrt{n}\abs{\frac 1 n \sum_{i = 1}^n f(Y_i) - \mu(f)} 
                    \, : \, 
                    f \in W_{1, q}(\mu)
                \right\}^3}
                - \E \norm{G}_{L^p(\mu)}^3
                }
                \leq n^{-1/2} b(n,0,1). 
            \end{align*}

    \end{itemize}

Let us consider the same notations as in \cite{DM2007}: 
\begin{nota}
    Define the function $F_\mu$ by: $F_\mu(x) = \mu(]0,x])$ if $x \geq 0$ and $F_\mu(x) = -\mu([x,0[)$ if $x \leq 0$. Define also the nonnegative random variable $Y_{p,\mu} = \abs{F_\mu (Y_0)}^{1/p}$. 
\end{nota}

Since $L^p(\mu)$ is 2-smooth, $p\geq 2$, as a consequence of Proposition \ref{prp:2smooth+tau} and Lemma \ref{cor:mixing_beta} we get in particular the following results.
\begin{cor}
    \label{Lp:cor_beta_Ypmu_p2}
    Let $\delta \in ]0,1]$. Assume that $\E Y_{2,\mu}^{2+\delta} < \infty$ and 
    $
        \sum_{k \geq 1} k \int_0^{\beta_{2,Y}(k)} Q_{Y_{2,\mu}}^{2+\delta} (u) \, du < \infty. 
    $
    Then $W_{1,2}(\mu)$ is a Donsker class for $(Y_n)_n$ and for any $M \geq 0$, 
        $$\sup_{f \in \Lambda_{2+\delta}(\mathbb B, M)}\Delta_n(f) = O(n^{-\delta/2}).$$
    Moreover,
    \begin{align*}
        \abs{\E \bigp{\sup\left\{
                    \sqrt{n}\abs{\frac 1 n \sum_{i = 1}^n f(Y_i) - \mu(f)} 
                    \, : \, 
                    f \in W_{1, 2}(\mu)
                \right\}^{2}}
                - \E \norm{G}_{L^2(\mu)}^{2}
                }
                = O (n^{-\delta/2}).
    \end{align*}
\end{cor}

\begin{cor}
    \label{Lp:cor_beta_Ypmu}
    Let $p >2$ and $r = \min(p,3)$. Assume that $\E Y_{p,\mu}^{r} < \infty$ and 
    \begin{align}
        \label{Lp:cond_beta}
        \sum_{k \geq 1} k \int_0^{\beta_{2,Y}(k)} Q_{Y_{p,\mu}}^{r} (u) \, du < \infty. 
    \end{align}
    Then $W_{1,q}(\mu)$ is a Donsker class for $(Y_n)_n$ and 
    $$\sup_{f \in \Lambda_r(L^p(\mu), M)}\Delta_n(f)= O(n^{-(r-2)/2}).$$
    Moreover,
    \begin{align*}
        \abs{\E \bigp{\sup\left\{
                    \sqrt{n}\abs{\frac 1 n \sum_{i = 1}^n f(Y_i) - \mu(f)} 
                    \, : \, 
                    f \in W_{1, q}(\mu)
                \right\}^{r}}
                - \E \norm{G}_{L^p(\mu)}^{r}
                }
                = O (n^{-(r-2)/2}).
    \end{align*}
\end{cor}

Let us only prove Corollary \ref{Lp:cor_beta_Ypmu}. 
\begin{proof}
    First, by definition of the $\beta$-mixing coefficients, $\beta_{2,X}(k) \leq \beta_{2,Y}(k)$ for any $k \geq 1$.
    Now, note that $Q_{\norm{X_0}_{L^p(\mu)}} \leq Q_{Y_{p,\mu}} + \E Y_{p, \mu}$. Hence 
    \begin{align*}
        \int_0^{\beta_{2,Y}(n)} Q_{\norm{X_0}_{L^p(\mu)}}^r(u) du 
        &\leq 
        2^{r-1} \bigp{\int_0^{\beta_{2,Y}(n)} Q_{Y_{p, \mu}}^r(u) du + \beta_{2,Y}(n)(\E Y_{p,\mu})^r}
        \\
        &\leq 
        2^{r-1} \bigp{\int_0^{\beta_{2,Y}(n)} Q_{Y_{p, \mu}}^r(u) du + \beta_{2,Y}(n)\E Y_{p,\mu}^r}. 
    \end{align*}
    Since $\E Y_{p,\mu}^r = \int_0^1 Q_{Y_{p, \mu}^r}(u) du$ and $Q_{Y_{p, \mu}^r}$ is a non-increasing function, and using \parencite[Lemma 2.1]{Rio2017}{}, we get
    \begin{align*}
        \int_0^{\beta_{2,Y}(n)} Q_{\norm{X_0}_{L^p(\mu)}}^r(u) du 
        \leq 
        2^r \int_0^{\beta_{2,Y}(n)} Q_{Y_{p, \mu}}^r(u) du. 
    \end{align*}
    Therefore condition (\ref{cor:mixing_beta_condition}) is derived from (\ref{Lp:cond_beta}) and the result follows combining Lemma \ref{cor:mixing_beta} and the discussion before the statement of Corollary \ref{Lp:cor_beta_Ypmu_p2}. 
\end{proof}

\subsection{Application to empirical processes associated with intermittent maps.}

For $\gamma \in ]0,1[$, let $T_\gamma: [0,1] \to [0,1]$ be the intermittent map defined by \cite{LSV1999} as follows: 
\begin{align*}
    T_\gamma(x) = \left\{
    \begin{array}{ll}
        x(1+2^\gamma x^\gamma) & \text{if } \, x \in [0, 1/2[  \\
        2x - 1 &  \text{if } \, x \in [1/2, 1]
    \end{array}
    \right. .
\end{align*}
As shown in \cite{LSV1999}, for all $\gamma \in ]0,1[$, there exists a unique absolutely continuous $T_\gamma$-invariant probability measure $\nu_\gamma$ on [0,1] whose density $h_\gamma$ satisfies: there exist two finite constants $c_1,c_2 > 0$ such that for all $x \in [0,1]$, $c_1 \leq x^\gamma h_\gamma(x) \leq c_2$. 
If there is no confusion, we will omit the index $\gamma$ for the sake of clarity. 
Let us fix $\gamma$ and consider $K$ the Perron-Frobenius operator of $T$ with respect to $\nu$ defined by 
\begin{align*}
    \nu(f\circ T.g) = \nu(f.K g) , \; \textrm{ for any }f,g \in \mathbb L^2(\nu).
\end{align*}
Then, by considering $(Z_i)_{i \in \mathbb Z}$  a stationary Markov chain with invariant measure $\nu$ and transition kernel $K$, for any positive integer $n$, on the probability space $([0,1], \nu)$, $(T, T^2, \cdots, T^n)$ is distributed as $(Z_n, Z_{n-1}, \cdots, Z_1)$ on a probability space $(\Omega, \mathcal A, \Proba)$ (see for instance Lemma XI.3 in \cite{HH2001}). Consequently, the two following empirical processes have the same distribution 
\begin{itemize}
    \item $\left\{G_n(t) = \frac{1}{\sqrt n}\sum_{k=1}^n [\indic_{T^k \leq t} - F(t)] \, ; \, t \in [0,1]\right\}$
    \item $\left\{L_n(t) = \frac{1}{\sqrt n}\sum_{k=1}^n [\indic_{Z_k \leq t} - F(t)] \, ; \, t \in [0,1]\right\}$
\end{itemize}
where $F(t) = \nu([0,t])$.

\begin{cor}
    For any $\gamma \in ]0,1/3[$, 
    \begin{align*}
        \{ \, G_n(t) \, : t \in [0,1]\} \xrightarrow[n \to \infty]{\law} \{G(t) \,:\, t \in [0,1]\} \; \textrm{in } L^2([0,1]),
    \end{align*} 
    and for any $M \geq 0$ and any $\delta \in ]0,1]$, 
    $$
        \sup_{f \in \Lambda_{2+\delta}(L^2([0,1]), M)}\Delta_n(f) = O(n^{-\delta/2}). 
    $$
    Moreover,
    \begin{align*}
        \abs{\E \norm{G_n}_{L^2([0,1])}^{2}
                - \E \norm{G}_{L^2([0,1])}^{2}
                }
                = O (n^{-1/2}).
    \end{align*}
\end{cor}

\begin{cor}
    Let $p>2$ and $r = \min(p, 3)$.
    For any $\gamma \in ]0,1/3[$, 
    \begin{align*}
        \{ \, G_n(t) \, : t \in [0,1]\} \xrightarrow[n \to \infty]{\law} \{G(t) \,:\, t \in [0,1]\} \; \textrm{in } L^p([0,1]),
    \end{align*} 
    and for any $M \geq 0$, 
    $$
        \sup_{f \in \Lambda_{r}(L^p([0,1]), M)}\Delta_n(f) = O(n^{-(r-2)/2}). 
    $$
    Moreover,
    \begin{align*}
        \abs{\E \norm{G_n}_{L^p([0,1])}^r
                - \E \norm{G}_{L^p([0,1])}^{r}
                }
                = O (n^{-(r-2)/2}).
    \end{align*}
\end{cor}

\begin{proof}     
    Since we are interested in the behavior of the partial sum, we can work on $(Z_n)_{n \geq 1}$ rather than $(T^n)_{n \geq 1}$. Let us denote $X_k = \{\indic_{Z_k \leq t} - F(t) : t \in [0,1]\}$ which takes values in $L^p(\nu)$ for any $p \geq 2$. 
    Since $\nu$ is supported on $[0,1]$, condition (\ref{prp:mixing_tau_condition}) of Proposition \ref{prp:2smooth+tau} reads as $\sum_{k \geq 1} k\tau_{2,X}(k) < \infty$.
    
    Let us estimate $\tau_{2,X}(k)$ in order to apply Proposition \ref{prp:2smooth+tau}. To this aim, we would apply results of \cite{DP2009} and \cite{DM2015} whose main argument is the modelling of the dynamical system by Young towers. We refer to \parencite[Section 4.1]{DM2015} for the construction of the tower $X$ associated to $T$ and for the mappings $\pi$ from $X$ to $[0, 1]$ and $F$ from $X$ to $X$ such that $T \circ \pi = \pi \circ F$. On $X$ there is a probability measure $m_0$ and a unique $F$-invariant probability measure $\bar \nu$ with density $h_0$ with respect to $m_0$. Note that the unique $T$-invariant probability measure $\nu$ is then given by $\nu = \bar \nu \circ \pi$ and denote $P$ the Perron-Frobenius operator of $F$ with respect to $\bar \nu$. Moreover, there exists a distance $d$ on $X$ such that for any $x,y \in X$, $d(x, y) \leq 1$ and $|\pi(x) - \pi(y)| \leq \kappa d(x, y)$ for some positive constant $\kappa$.
    \\

    \noindent
    For any $k,l \geq 0$, let us consider $(Z_k^{*}, Z_{k+l}^{*})$ the coupling associated with
    \begin{align*}
        \tau_{\abs{.}^{1/p}}(\calF_0, (Z_k, Z_{k+l}))
        \eqdef 
        \frac 1 2
        \nu\bigp{\sup \left\{
            \abs{\condEc{f(Z_k, Z_{k+l})}{Z_0} - \E\bigcro{f(Z_k, Z_{k+l})}}
            : 
            f \in \Lambda_1(\R ^2, \abs{.}^{1/p})
        \right\}
        } 
    \end{align*}
    that is 
    \begin{align*}
        \inf \left\{ 
            \E \abs{Z_k - Z_k'}^{1/p} + \E \abs{Z_{k+l} - Z_{k+l}'}^{1/p}
            \, : \, 
            (Z_k', Z_{k+l}') \overset{\law}{=} (Z_k, Z_{k+l}), (Z_k', Z_{k+l}') \indep Z_0
        \right\}
        \hspace{45pt}
        \\
        =
        \E \abs{Z_k - Z_k^*}^{1/p} + \E \abs{Z_{k+l} - Z_{k+l}^*}^{1/p}
        .
    \end{align*}
    Let $X_k^{*} = \{\indic_{Z_k^* \leq t} - F(t) : t \in [0,1]\}$ and $X_{k+l}^{*}=\{\indic_{Z_{k+l}^* \leq t} - F(t) : t \in [0,1]\}$. By definition of $\tau_{2,X}(k)$, 
    \begin{align}
        \tau_{2,X}(k) 
        &\leq \underset{l \geq 0}{\sup} \frac{1}{2} \E_\nu \bigp{\norm{X_k - X_k^{*}}_{L^p([0,1])}+\norm{X_{k+l} - X_{k+l}^{*}}_{L^p([0,1])}}
            \nonumber
        \\
        &\leq 
        \underset{l \geq 0}{\sup} \frac{1}{2} \E_\nu \bigp{\abs{Z_k - Z_k^{*}}^{1/p}+\abs{Z_{k+l} - Z_{k+l}^{*}}^{1/p}}
            \nonumber
        \\ 
        & \leq 
        \underset{l \geq 0}{\sup} \, \tau_{\abs{.}^{1/p}}(\calF_0, (Z_k, Z_{k+l})). 
            \label{Lp_intermittent:tauX_control}
    \end{align}

    Proceeding in a similar way as in the proof of Theorem 2.1 in \cite{DP2009} and taking into account the inequality (4.3) in \cite{DM2015} rather than Lemma 2.3 in \cite{DP2009}, we derive the following lemma whose proof will be done in Section \ref{section:proofs}. 
    \begin{lem}
            \label{lem:Young_tau_estimate}
        There exists a positive constant $c$ such that for any $k,l \geq 0$, $$\tau_{\abs{.}^{1/p}}(\calF_0, (Z_k, Z_{k+l})) \leq ck^{-(1-\gamma)/\gamma}.$$ 
    \end{lem}

    Combining Lemma \ref{lem:Young_tau_estimate} with (\ref{Lp_intermittent:tauX_control}), we get $\tau_{2, X}(k) \leq ck^{-(1-\gamma)/\gamma}$. 
    Hence $\sum_{k \geq 1}k \tau_{2,X}(k) < \infty$ provided that $\gamma  < 1/3$. 
    The result follows from Proposition \ref{prp:2smooth+tau}.
\end{proof}

\section{Proofs}\label{section:proofs}

\subsection{Proof of Theorem \ref{theoreme:vitesse}}

\begin{lem}
        \label{lemme:egalite_bilineaire}
    Under conditions  \ref{hyp:conv_gamma2}, \ref{hyp:CLT} and \ref{hyp:u.i.}, for any symmetric bilinear continuous form $\varphi$, 
    \begin{align*}
       \E\bigcro{\varphi(G,G)} = \E\bigcro{\varphi(X_0, X_0)} +2\sum_{k \geq 1} \E\bigcro{\varphi(X_0, X_k)}.
    \end{align*}
\end{lem}

\begin{proof}[Proof of Lemma \ref{lemme:egalite_bilineaire}]
    From \ref{hyp:CLT} and \ref{hyp:u.i.}, since $\varphi$ is continuous, 
    \begin{align*}
        \E\bigcro{\varphi(G,G)} 
        = 
        \lim\limits_{n \to +\infty} \E\bigcro{\varphi\bigp{\frac{S_n}{\sqrt n}, \frac{S_n}{\sqrt n}}}. 
    \end{align*}
    By stationarity, for any $n > 1$, 
    \begin{align*}
        \E\bigcro{\varphi\bigp{\frac{S_n}{\sqrt n}, \frac{S_n}{\sqrt n}}}
        & 
        = \E\bigcro{\varphi(X_0, X_0)} + \frac 2 n \sum_{i = 1}^n \sum_{k = i}^n \E\bigcro{\varphi(X_i, X_k)}
        \\
        & 
        = \E\bigcro{\varphi(X_0, X_0)} + \frac 2 n \sum_{l = 1}^{n -1} \frac{n-l}{n} \E\bigcro{\varphi(X_0, X_l)}. 
    \end{align*}
    From \ref{hyp:conv_gamma2}, the dominated convergence theorem applies and infers that 
    \begin{align*}
        \lim\limits_{n \to +\infty} \E\bigcro{\varphi\bigp{\frac{S_n}{\sqrt n}, \frac{S_n}{\sqrt n}}}
        = 
        \E\bigcro{\varphi(X_0, X_0)} +2\sum_{k \geq 1} \E\bigcro{\varphi(X_0, X_k)}.
    \end{align*}
\end{proof}

\begin{proof}[Proof of Theorem \ref{theoreme:vitesse}]
We shall apply the Lindeberg method. In case of a triangular array of independent random variables taking values in a Banach space of type 2, this method has been recently used by Ra\v{c}kauskas and Suquet \cite{RS2023} to get rates in the central limit theorem (see their Theorem 21). 
Consider $(\eps_k)_{k\geq 1}$ a sequence of i.i.d. centered gaussian random variables distributed as $G$ and independent of $(X_k)_k$. 
Let $n \geq 1$. For any $i \leq n$, set $\Gamma_i = \sum_{k = i}^n \eps_k$. 
In order to simplify the calculation lines, write 
\begin{align*}
    & f^{(i)}_{j,k}(\cdot) = f^{(i)}\bigp{k^{-1/2}\bigcro{\,\cdot + \Gamma_j}}. 
\end{align*}

Applying Taylor's formula, 
\begin{align*}
    &\E\bigcro{f(n^{-1/2}S_n) - f(n^{-1/2}\Gamma_1)}
    = 
    \frac{1}{\sqrt n} \sum_{i =1}^n \E \bigcro{ f_{i+1,n}^{(1)}(S_{i-1})(X_i)}
    \\
    &\hspace{4cm}
    +
    \frac{1}{2n} \sum_{i = 1}^n \E \bigcro{ f_{i+1, n}^{(2)}(S_{i-1})(X_i, X_i) - f_{i+1, n}^{(2)}(S_{i-1})(\eps_i , \eps_i) }
    +
    \frac{1}{n} \sum_{i= 1}^n \E [R_i - \Tilde R_i],
\end{align*}
where 
 \begin{align}
        &
        \nonumber
        R_i = \int_0^1 (1-s) \bigcro{f_{i+1,n}^{(2)}\bigp{S_{i-1}+sX_i} - f_{i+1,n}^{(2)}(S_{i-1})}(X_i,X_i) \, ds
        \\ 
        \nonumber
        \text{and }&
        \\
        &
        \nonumber
        \Tilde R_i = \int_0^1 (1-s) \bigcro{ f_{i+1,n}^{(2)}\bigp{S_{i-1}+s\eps_i} - f_{i+1,n}^{(2)}(S_{i-1})}(\eps_i,\eps_i) \, ds.
    \end{align}
\noindent
 As by independence $\E [f_{i+1, n}^{(1)}(0)(X_i)] = \E [f_{i+1, n}^{(1)}(0)(\E X_i))] = 0$, we can write
 \begin{align*}
     &\frac{1}{\sqrt n} \sum_{i =1}^n \E \bigcro{ f_{i+1,n}^{(1)}(S_{i-1})(X_i)}
        = 
     \frac{1}{\sqrt n} \sum_{i =1}^n \sum_{k = 1}^{i-1} \E \bigcro{ f_{i+1, n}^{(1)}(S_{i-k})(X_i) - f_{i+1, n}^{(1)}(S_{i-k-1})(X_i) }
     \\
     & \hspace{22pt}
        = 
     \frac{1}{n} \sum_{i = 1}^n \sum_{k=1}^{i-1} \E \bigcro{ f_{i+1, n}^{(2)}(S_{i-k-1})(X_{i-k}, X_i) } 
     \\
     &\hspace{44pt}
     +
     \frac{1}{n} \sum_{i = 1}^n \sum_{k=1}^{i-1} \int_0^1 \E \bigcro{
                                                            \left\{f_{i+1, n}^{(2)}(S_{i-k-1}+tX_{i-k}) - f_{i+1, n}^{(2)}(S_{i-k-1}) \right\}(X_{i-k}, X_i) 
                                                                    } dt
    . 
 \end{align*}
 Furthermore, according to Lemma \ref{lemme:egalite_bilineaire} and since $f_{i+1, n}^{(2)}(S_{i-1})$ is symmetric, for $(X_i^*)_i$ an independent copy of $(X_i)_i$, independent of $(\eps_i)_i$, we can write 
 \begin{align*}
     \E \bigcro{
     f_{i+1, n}^{(2)}(S_{i-1})(\eps_i , \eps_i) 
     }
     = 
     \E \bigcro{
     f_{i+1, n}^{(2)}(S_{i-1})(X_0^*, X_0^*) 
     }
     + 
     2 \sum_{k\geq 1} \E \bigcro{
     f_{i+1, n}^{(2)}(S_{i-1})(X_0^*, X_k^*) 
     }
     . 
 \end{align*}
\noindent
 Thus,
\begin{align}
    \E\bigcro{f(n^{-1/2}S_n) - f(n^{-1/2}\Gamma_1)} = D_1 - D_2 + D_3 + D_4 + D_5
    \label{generalbanachcase:decomposition}
\end{align}
where  
 \begin{align*}
    &D_1
    = 
    \frac{1}{n} \sum_{i = 1}^n \sum_{k=1}^{i-1} \E \bigcro{ f_{i+1, n}^{(2)}(S_{i-k-1})(X_{i-k}, X_i) - f_{i+1, n}^{(2)}(S_{i-k-1})(X_{i-k}^*, X_i^*)} 
    \\
    &\hspace{4cm}
    +
    \frac{1}{n} \sum_{i = 1}^n \sum_{k=1}^{i-1} \E \bigcro{ f_{i+1, n}^{(2)}(S_{i-k-1})(X_{i-k}^*, X_i^*) - f_{i+1, n}^{(2)}(S_{i-1})(X_0^* , X_k^*) },
    \\
    &D_2 = 
    \frac{1}{n} \sum_{i = 1}^n \sum_{k \geq i} \E \bigcro{f_{i+1, n}^{(2)}(S_{i-1})(X_0^* , X_k^*) }, 
    \\
    &D_3 = 
    \frac{1}{2n} \sum_{i = 1}^n \sum_{k=1}^{i-1} \E \bigcro{ f_{i+1, n}^{(2)}(S_{i-1})(X_i,X_i) - f_{i+1, n}^{(2)}(S_{i-1})(X_0^* , X_0^*) }, 
    \\
    &D_4 =
    \frac{1}{n} \sum_{i = 1}^n \sum_{k=1}^{i-1} \int_0^1 \E \bigcro{
                                                            \left\{f_{i+1, n}^{(2)}(S_{i-k-1}+tX_{i-k}) - f_{i+1, n}^{(2)}(S_{i-k-1}) \right\}(X_{i-k}, X_i) 
                                                                    } dt
    \\
    &D_5 =
    \frac{1}{n} \sum_{i= 1}^n \E [R_i - \Tilde R_i]
    .
 \end{align*}
 ~
 \\[0.5cm]
 Let us deal with $D_1$. 
 Taking into account stationarity and independence, 
 \begin{align*}
     \hspace{40pt}&\hspace{-40pt}
     \E \bigcro{ f_{i+1, n}^{(2)}(S_{i-k-1})(X_{i-k}, X_i) - f_{i+1, n}^{(2)}(S_{i-k-1})(X_{i-k}^*, X_i^*)} 
     \\
     &=
     \sum_{l = 1}^{i-k-1} \E \left[ \{f_{i+1, n}^{(2)}(S_{i-k-l}) - f_{i+1, n}^{(2)}(S_{i-k-l-1})\}(X_{i-k}, X_i) \right.
     \\
     &\hspace{5cm}
     \left.
     - \{f_{i+1, n}^{(2)}(S_{i-k-l}) - f_{i+1, n}^{(2)}(S_{i-k-l-1})\}(X_{i-k}^*, X_i^*) \right] 
     \\
     &=
     \sum_{l = 1}^{i-k-1} \E \bigp{ \condEc{A_{i,k,l}(X_{0}, X_k) -A_{i,k,l}(X_0^*, X_k^*)}{\calF_0} }
 \end{align*}
 where $A_{i,k,l} = f_{i+1, n}^{(2)}\bigp{\sum_{j = 1}^{i-k-l} X_{j - (i-k)}} - f_{i+1, n}^{(2)}\bigp{\sum_{j = 1}^{i-k-l-1} X_{j - (i-k)}}$  is a continuous bilinear form whose norm is bounded by $n^{-\delta /2}\normB{X_{-l}}^\delta$ since $f \in \Lambda_{2+\delta}(\mathbb B, M)$. Hence, 
 \begin{align*}
     \Delta
     \eqdef
     \abs{
        \frac{1}{n} \sum_{i = 1}^n \sum_{k=1}^{i-1} \E \bigcro{ f_{i+1, n}^{(2)}(S_{i-k-1})(X_{i-k}, X_i) - f_{i+1, n}^{(2)}(S_{i-k-1})(X_{i-k}^*, X_i^*)} 
     }
     \leq 
     n^{-\delta /2} \sum_{k = 1}^{n-1} \sum_{l=1}^{n-k-1} a(k). 
 \end{align*}
 \\
 On another hand
 \begin{align*}
     \Delta 
     = 
     \abs{\sum_{l = 1}^{i-k-1} \E \bigp{ \condEc{B_{i,k,l}(X_l, X_{l+k}) -B_{i,k,l}(X_l^*, X_{l+k}^*)}{\calF_0} }}
 \end{align*}
 where 
 $
    B_{i,k,l} = f_{i+1, n}^{(2)}\bigp{\sum_{j = 1}^{i-k-l} X_{j - (i-k-l)}} - f_{i+1, n}^{(2)}\bigp{\sum_{j = 1}^{i-k-l-1} X_{j - (i-k-l)}}
$  
is a continuous bilinear form whose norm is bounded by $n^{-\delta /2}\normB{X_{0}}^\delta$ since $f \in \Lambda_{2+\delta}(\mathbb B, M)$. 
Thus, 
 \begin{align*}
     \Delta
     \leq 
     n^{-\delta /2} \sum_{k = 1}^{n-1} \sum_{l=1}^{n-k-1} b(l). 
 \end{align*}
Combining both controls of $\Delta$, we get 
\begin{align*}
    \Delta
    \leq 
    n^{-\delta /2} \sum_{k=1}^{n-1} \sum_{l = 1}^{n-k-1} \min(\gamma_{2,\delta}(k),\gamma_{2,\delta}(l))
    \leq 
    n^{-\delta /2} \sum_{k=1}^{n-1} k \gamma_{2,\delta}(k). 
\end{align*}

\noindent
 On another hand, taking into account independence and stationarity,
 \begin{align*}
     \hspace{4cm}&\hspace{-4cm} 
     \Delta^* \eqdef
     \abs{\frac{1}{n} \sum_{i = 1}^n \sum_{k=1}^{i-1} \E \bigcro{ f_{i+1, n}^{(2)}(S_{i-k-1})(X_{i-k}^*, X_i^*) - f_{i+1, n}^{(2)}(S_{i-1})(X_0^* , X_k^*) }}
     \\
     &\leq 
     \frac{1}{n} \sum_{i = 1}^n \sum_ {k=1}^{i-1} n^{-\delta / 2}\abs{ \E \bigp{\normB{\sum_{j = i-k}^{i-1} X_j}^\delta \condEc{\Tilde A_{i,k,n}(X_0^*, X_k^*)}{\calF_0^*}}}
 \end{align*}
 where 
 $$
    \bigp{n^{-\delta / 2}\normB{\sum_{j = i-k}^{i-1} X_j}^\delta}\tilde A_{i,k,l} 
    = 
    f_{i+1, n}^{(2)}\bigp{S_{i-k-l}} - f_{i+1, n}^{(2)}\bigp{S_{i-k}}.
$$
Since $f \in \Lambda_{2+\delta}(\mathbb B, M)$, $\norm{\Tilde A_{i,k,l}}\leq 1$. Hence, 
\begin{align*}
    \Delta^* 
    \leq 
     n^{-\delta / 2} \sum_ {k=1}^{n-1} \E\bigp{\normB{S_k}^\delta} \E \bigp{\sup\limits_{\norm A \leq 1} \abs{ \condEp{ A(X_0, X_k)}{\calF_0}}}. 
\end{align*}
Since $\bigp{\normB{S_n}^2/n}_n$ is a uniformly integrable family, there exists $\lambda>0$ such that $\E \normB{S_k}^2 \leq \lambda k$ for any $k$. 
Applying Jensen's inequality, we get
 \begin{align*}
     \Delta^* 
        \leq 
     \lambda^{\delta/2} \, n^{-\delta/2} \sum_{k = 1}^n k^{\delta /2}\gamma(k). 
 \end{align*}
We derive 
\begin{align}
    \abs{D_1} \leq n^{-\delta /2} \sum_{k=1}^{n-1} k \gamma_{2,\delta}(k) + \lambda^{\delta/2} \, n^{-\delta/2} \sum_{k = 1}^n k^{\delta /2} \gamma(k).
    \label{generalbanachcase:D1}
\end{align}

 ~
 \\[0.5cm]
 We turn to the control of $D_2$. 
Since $f \in \Lambda_{2+\delta}(\mathbb B, M)$,
 \begin{align*}
     \norm{f_{i+1, n}^{(2)}(S_{i-1})} 
     \leq \norm{f_{i+1, n}^{(2)}(S_{i-1}) - f^{(2)}(0)} + \norm{f^{(2)}(0)} 
     \leq n^{-\delta / 2}\normB{S_{i-1} + \Gamma_{i+1}}^\delta + M,
 \end{align*}
 with 
 \begin{align*}
     \E (\normB{S_{i-1} + \Gamma_{i+1}}^\delta) 
     &\leq \bigp{[\lambda(i-1)]^{1/2} + [(n-i)\E\normB{G}^2]^{1/2}}^\delta 
     \\
     &\leq [\lambda^{\delta/2}+(\E\normB{G}^2)^{\delta/2}].[(i-1)^{1/2}+(n-i-1)^{1/2}]^\delta  
     \\
     &\leq 2^\delta(\lambda^{\delta/2}+(\E\normB{G}^2)^{\delta/2})n^{\delta/2}
     . 
 \end{align*}
 By independence between $(X_i)_i$ and $(X_i^*)_i$, we derive
 \begin{align}
     \abs{D_2}
     &\leq
        \frac{1}{n} \sum_{i = 1}^n \sum_{k \geq i} \E \bigp{n^{-\delta / 2}\normB{S_{i-1} + \Gamma_{i+1}}^\delta + M} \E\bigp{\sup\limits_{\norm A \leq 1} \abs{\condEp{A(X_0^*, X_k^*)}{\calF_0^*}}}
    \nonumber 
    \\
     &\leq 
         (2^\delta\lambda^{\delta/2}+2^\delta(\E\normB{G}^2)^{\delta/2} + M) \frac 1 n \, \sum_{k \geq 1}\min(n,k) \E\bigp{\sup\limits_{\norm A \leq 1} \abs{\condEp{A(X_0, X_k)}{\calF_0}}}
     \nonumber
     \\
     &\leq 
     (2^\delta\lambda^{\delta/2}+2^\delta(\E\normB{G}^2)^{\delta/2} + M) \, n^{-\delta / 2} \, \sum_{k \geq 1} k^{\delta /2} \gamma(k). 
     \label{generalbanachcase:D2}
 \end{align}

 ~
 \\[0.5cm]
 We can write $D_3$ as
 \begin{align*}
     \frac{1}{2n} \sum_{i = 1}^n \sum_{k = 1}^{i-1} \E\bigcro{
        \{ f_{i+1, n}^{(2)}(S_{i-k}) - f_{i+1, n}^{(2)}(S_{i-k-1}) \} (X_i, X_i) - \{ f_{i+1, n}^{(2)}(S_{i-k}) - f_{i+1, n}^{(2)}(S_{i-k-1}) \} (X_0^*, X_0^*)
     }
    \\
     = 
     \frac{1}{2n} \sum_{i = 1}^n \sum_{k = 1}^{i-1} \E\bigp{ \condEc{
        A_{i,k} (X_k, X_k) - A_{i,k} (X_0^*, X_0^*)
    }{\calF_0}
     }
 \end{align*}
 where $A_{i,k} = f_{i+1,n}^{(2)}\bigp{\sum_{j = 1}^{i-k} X_{j - (i-k)} } - f_{i+1,n}^{(2)}\bigp{\sum_{j = 1}^{i-k-1} X_{j - (i-k)} }$ is a continuous bilinear form whose norm is bounded by $n^{-\delta /2}\normB{X_0}^\delta$. 
 \\
 Thus,  
 \begin{align}
    \abs{ D_3 }
    \leq 
        \frac{n^{-\delta /2}}{2} \sum_{k = 1}^n \E \bigp{ \normB{X_0}^\delta \sup\limits_{\norm A \leq 1} \abs{ \condEc{A(X_k, X_k)}{\calF_0} - A(X_0, X_0) } }
    \leq 
        \frac{n^{-\delta /2}}{2} \sum_{k = 1}^n b(k)
    . 
    \label{generalbanachcase:D3}
 \end{align}
~
 \\[0.5cm]
 We turn to the control of $D_4$. By stationarity
 \begin{align*}
     D_4
     =
     \frac{1}{n}\sum_{i = 1}^n \sum_{k =1}^{i-1} \int_0^1 \E\bigp{\condEc{C_{i,k,t}(X_0,X_k)}{\calF_0}} dt 
 \end{align*}
 where $C_{i,k,t} = f_{i+1,n}^{(2)}\bigp{\sum_{j=1}^{i-k-1}X_{j-(i-k)}+tX_0}-f_{i+1,n}^{(2)}\bigp{\sum_{j=1}^{i-k-1}X_{j-(i-k)}}$ is a continuous bilinear form whose norm is bounded by $t^\delta n^{-\delta /2}\normB{X_0}^\delta$. Consequently, 
 \begin{align*}
     \abs{D_4} \leq n^{-\delta /2} \sum_{k = 1}^{n-1} \E\bigp{\normB{X_0}^\delta \sup\limits_{\norm{A}\leq 1} \abs{\condEc{A(X_0, X_k)}{\calF_0}}}. 
 \end{align*}
 Now, 
 \begin{align*}
     \E\bigp{\normB{X_0}^\delta \sup\limits_{\norm{A}\leq 1} \abs{\condEc{A(X_0, X_k)}{\calF_0}}} 
     &\leq a(k) + \E(\normB{X_0}^\delta)\gamma(k)
     \leq a(k) + \lambda^{\delta/2}\gamma(k). 
 \end{align*}
 Hence
 \begin{align}
     \abs{D_4} \leq n^{-\delta /2}\sum_{k=1}^{n-1}\gamma_{2,\delta}(k) + \lambda^{\delta/2}n^{-\delta/2}\sum_{k = 1}^{n-1}\gamma(k).
     \label{generalbanachcase:D4}
 \end{align}
 
 It remains to control $D_5$. 
 Taking into account $\delta$-Hölder continuity, 
 \begin{align}
     \abs{D_5} \leq n^{-\delta / 2}\bigp{ \E \normB{X_0}^{2+\delta} + \E \normB{G}^{2+\delta} }. 
     \label{generalbanachcase:D5}
 \end{align}
 \\[0.5cm]
 Combining decomposition (\ref{generalbanachcase:decomposition}) with the upper bounds obtained in (\ref{generalbanachcase:D1}), (\ref{generalbanachcase:D2}), (\ref{generalbanachcase:D3}), (\ref{generalbanachcase:D4}) and (\ref{generalbanachcase:D5}), we derive
 \begin{align*}
    \hspace{50pt}&\hspace{-50pt}
    \abs{\E\bigcro{f(n^{-1/2}S_n) - f(n^{-1/2}\Gamma_1)}}
        \\
        \leq & n^{-\delta /2} 
    \bigp{
    \sum_{k=1}^{n-1} (k+2) \gamma_{2,\delta}(k) 
        +
    (c_\delta + M) \sum_{k \geq 1} k^{\delta /2} \gamma(k)
        +
    \E \normB{X_0}^{2+\delta} + \E \normB{G}^{2+\delta}
    }
 \end{align*}
 with $c_\delta = (2^\delta+2)\lambda^{\delta/2} + 2^\delta(\E \normB{G}^2)^{\delta/2}$, which proves the theorem.

\end{proof}

\subsection{Proof of Lemma \ref{prp:mixing_tau}}

    Here we control the coefficients involved in Theorem \ref{theoreme:vitesse}. To this end, we apply a coupling lemma for $\tau$-mixing coefficients. 
\\
First, according to \parencite[Lemma 1]{DM2006} with $d(x,y) = \normB{x-y}$, for any $n$ there exists a random variable $X_n^*$ distributed as $X_n$, independent of $\calF_0$ and such that     
$$\E(\normB{X_n - X_n^*}) = \tau_1(n).$$
On the other hand according to \parencite[Lemma 1]{DM2006} with $d((x,x'),(y,y')) = \normB{x-y}+\normB{x'-y'}$, for any $k,l$ there exists  $(X_l^{**},X_{k+l}^{**})$ distributed as $(X_l,X_{k+l})$, independent of $\calF_0$ and such that 
$$\frac{1}{2}\bigcro{\E\bigp{\normB{X_l - X_l^{**}} + \normB{X_{l+k} - X_{l+k}^{**}}}} = \tau\left(\calF_0, (X_l, X_{l+k})\right).$$ 
In what follows, all the suprema below are taken over bilinear continuous forms whose norm is bounded by 1. 
\\
Let us begin by the control of $\gamma(k)$. 
For any $k \geq 0$, applying Proposition 1 in \cite{DD2003},
\begin{align*}
    \E \bigp{\sup\abs{ \condEc{ A(X_0, X_k)}{\calF_0}}}
    &= 
    \E \bigp{\sup\abs{ \condEc{ A(X_0, X_k-X_k^*)}{\calF_0}}}
    \\
    &\leq 
    \E(\normB{X_0}\normB{X_k-X_k^*})
    \\
    &\leq 
    \int_0^{\tau_{{}_1}(k)} Q_{\normB{X_0}}\circ G_{\normB{X_k - X_k^*}}(u) \, du 
    .
\end{align*}
Hence, taking into account properties of the function $G$ and stationarity, 
\begin{align}
        \label{prooftau_gamma}
    \gamma(k)
    \leq 
    2 \int_0^{\tau_{{}_1}(k)/2} Q_{\normB{X_0}} \circ G_{\normB{X_0}} (v) \, dv. 
\end{align}

It remains to control $\gamma_{2,\delta}(k)$, to this end let us bound $a(k)$ in one hand and $b(l)$ on the other hand. 
For any $k \geq 0$ and $l \geq 0$, 
\begin{align*}
    \E \bigp{
        \normB{X_{-l}}^\delta \sup \abs{
            \condEc{A(X_0, X_k)}{\calF_0} - \E[A(X_0, X_k)]
        }
    }
    \hspace{-5.4cm}&
    \\
    &\leq 
    \E \bigp{
        \normB{X_{-l}}^\delta \sup \left\{
            \condEc{\abs{A(X_0, X_k-X_k^*)}}{\calF_0} + \E[\abs{A(X_0, X_k-X_k^*)}
    \right\}
    }
    \\
    &\leq 
    2\, \E\bigp{
        \normB{X_{-l}}^\delta \normB{X_0} \normB{X_k - X_k^*}
    }
    \\
    &\leq 
    2\, \int_0^{\tau_{{}_1}(k)} Q_{\normB{X_{-l}}^\delta \normB{X_0}} \circ G_{\normB{X_k - X_k^*}} (u) \, du
    \\
    &\leq 
    4\, \int_0^{\tau_{{}_1}(k)/2}  Q_{\normB{X_{-l}}^\delta \normB{X_0}} \circ G_{\normB{X_0}} (v) \, dv
    . 
\end{align*}
Hence, taking the supremum over $l$ and applying \parencite[Lemma 2.1]{Rio2017}, 
\begin{align}
        \label{prooftau_a}
    a(k) \leq 4 \, \int_0^{\tau_{{}_1}(k)/2} Q_{\normB{X_0}}^{1+\delta} \circ G_{\normB{X_0}} (v) \, dv
    .
\end{align}
Now, for any $k,j \geq 0$, 
\begin{align*}
    \E\bigp{
        \normB{X_0}^\delta \sup\abs{\condEc{A(X_k,X_{k+j})}{\calF_0} - \E\bigcro{A(X_k, X_{k+j}}} 
    }
    \hspace{-5.7cm}&
    \\
    &=
    \E\bigp{
        \normB{X_0}^\delta \sup\abs{\condEc{A(X_k - X_k^{**},X_{k+j}) + A(X_k^{**}, X_{k+j} - X_{k+j}^{**})}{\calF_0} 
    }
    }
    \\
    &\leq 
    \E\bigp{\normB{X_0}^\delta \normB{X_k - X_k^{**}} \normB{X_{k+j}}}
    +
    \E\bigp{\normB{X_0}^\delta \normB{X_k^{**}} \normB{X_{k+j} - X_{k+j}^{**}} }
    \\
    &\leq 
    \int_0^{\E\normB{X_k - X_k^{**}}} Q_{\normB{X_0}^\delta \normB{X_{k+j}}} \circ G_{\normB{X_k - X_k^{**}}} (u) \, du 
    \\
    &\hspace{2cm}+
    \int_0^{\E\normB{X_{k+j} - X_{k+j}^{**}}} Q_{\normB{X_0}^\delta \normB{X_{k}^{**}}} \circ G_{\normB{X_{k+j} - X_{k+j}^{**}}} (u) \, du 
    .
\end{align*}
As $\min(\E\normB{X_k - X_k^{**}},\E \normB{X_{k+j} - X_{k+j}^{**}} ) \leq (\E\normB{X_k - X_k^{**}} + \E \normB{X_{k+j} - X_{k+j}^{**}})/2$, we infer that 
\begin{align*}
    \E\bigp{
        \normB{X_0}^\delta \sup\abs{\condEc{A(X_k,X_{k+j})}{\calF_0} - \E\bigcro{A(X_k, X_{k+j}}} 
    }
    \hspace{-3.7cm}&
    \\
    &\leq 
    4\int_0^{\tau(\calF_0, (X_k, X_{k+j}))/2} Q_{\normB{X_0}}^{1+\delta}(u) \circ G_{\normB{X_0}} (v) \, dv. 
\end{align*}
Hence, 
\begin{align}
        \label{prooftau_b}
    b(k) \leq 4\int_0^{\tau_{{}_2}(k)/2} Q_{\normB{X_0}}^{1+\delta}(u) \circ G_{\normB{X_0}} (v) \, dv.
\end{align}
Combining (\ref{prooftau_gamma}), (\ref{prooftau_a}) and (\ref{prooftau_b}) with the fact that $\tau_1(k) \leq \tau_2(k)$, the result follows. 

\subsection{Proof of Proposition \ref{prp:2smooth+tau}}

    As a first step, let us prove that under (\ref{prp:mixing_tau_condition}), $(\condEp{S_n}{\calF_0})_{n \geq 1}$ converges in $\mathbb L^2_{\mathbb B}$. To this end, let us prove that this sequence is a Cauchy sequence that is 
    \begin{align}
        \lim_{m \to \infty} \sup_{n >m} \norm{\sum_{i = m+1}^n \condEp{X_i}{\calF_0}}_{2,\mathbb B} = 0. 
            \label{condEp_Sn_F0:Cauchy}
    \end{align}
    Let $(X_i^*)_{i \geq 1}$ be the sequence constructed via the coupling lemma \parencite[Lemma 1]{DM2006}{} with $d(x,y) = \normB{x-y}$, i.e. $X_i^*$ distributed as $X_i$, independent of $\calF_0$ and such that $\E (\normB{X_i - X_i^*}) = \tau_1(i)$. Write
    \begin{align*}
        \norm{\sum_{i = m+1}^n \condEp{X_i}{\calF_0}}_{2,\mathbb B}^2
            \nonumber
        &= 
        \norm{\sum_{i = m+1}^n \condEp{X_i-X_i^*}{\calF_0}}_{2,\mathbb B}^2
        \\
        &\leq 
            \sum_{i = m+1}^n \E\bigp{\normB{X_i - X_i^*}\normB{\sum_{j = m+1}^n \condEp{X_j-X_j^*}{\calF_0}}}
        .
    \end{align*}
    Let us denote $Y_{n,m} \eqdef \sum_{i = m+1}^n \condEp{X_i-X_i^*}{\calF_0}$. Proceeding as in the proof of Proposition 1 in \parencite[]{DD2003}{} and using stationarity, 
    \begin{align*}
        \norm{Y_{n,m}}_{2,\mathbb B}^2 
        &\leq 
        \sum_{i = m+1}^n \int_0^{\tau_1(i)} Q_{Y_{n,m}} \circ G_{\normB{X_i - X_i^*}}(u) \, du 
        \\
        &\leq 
        2\sum_{i = m+1}^n \int_0^{G_{\normB{X_0}}(\tau_1(i)/2)} Q_{Y_{n,m}}(u) Q_{\normB{X_0}}(u) \, du 
        \\
        &\leq 
        2 \bigp{\int_0^1 Q_{Y_{n,m}}^2(u) du}^{1/2} \bigp{\int_0^1 Q_{\normB{X_0}}^2(u)\bigcro{\sum_{i = m+1}^n \indic_{u \leq G_{\normB{X_0}}(\tau_1(i)/2)}}^2 du }^{1/2}. 
    \end{align*}
    Hence, 
    \begin{align*}
        \norm{Y_{n,m}}_{2,\mathbb B} \leq 2\bigp{\int_0^1 Q_{\normB{X_0}}^2(u)\bigcro{\sum_{i = m+1}^n \indic_{u \leq G_{\normB{X_0}}(\tau_1(i)/2)}}^2 du }^{1/2}.
    \end{align*}
    Since
    \begin{align*}
        \bigcro{\sum_{i = m+1}^n \indic_{u \leq G_{\normB{X_0}}(\tau_1(i)/2)}}^2 
        &\leq 
        \sum_{i = m+1}^\infty (i+1)^2 \indic_{G_{\normB{X_0}}(\tau_1(i+1)/2) \leq u \leq G_{\normB{X_0}}(\tau_1(i)/2)}
        \\
        &\leq 
        2\sum_{i = m+1}^\infty \indic_{G_{\normB{X_0}}(\tau_1(i+1)/2) \leq u \leq G_{\normB{X_0}}(\tau_1(i)/2)}\sum_{k = 0}^{i} (k+1)
        \\
        &\leq 
        2\sum_{k=0}^\infty (k+1) \indic_{u \leq G_{\normB{X_0}}(\tau_1(m+1)/2)} \indic_{u \leq G_{\normB{X_0}}(\tau_1(k)/2)}
        , 
    \end{align*}
    we derive 
    \begin{align*}
        \norm{\sum_{i = m+1}^n \condEp{X_i}{\calF_0}}_{2,\mathbb B}
        &\leq 
        2^{3/2}\bigp{\sum_{i = 0}^\infty (i+1) \int_0^{\min(G_{\normB{X_0}}(\tau_1(i)/2),G_{\normB{X_0}}(\tau_1(m+1)/2))} Q_{\normB{X_0}}^2(u) du }^{1/2} 
            \nonumber
        \\
        &\leq 
        2^{3/2}\bigp{\sum_{i = 0}^\infty (i+1) \int_0^{\min(\tau_1(i),\tau_1(m+1))/2} Q_{\normB{X_0}} \circ G_{\normB{X_0}}(u) du }^{1/2}.
    \end{align*}
    Consequently, since (\ref{prp:mixing_tau_condition}) is verified,  (\ref{condEp_Sn_F0:Cauchy}) holds. 
    
    Recall that if $(\condEp{S_n}{\calF_0})_{n \geq 1}$ converges in $\mathbb L^2_{\mathbb B}$ then there exist a stationary sequence $(z_n)_{n \geq 0} \in \mathbb L^2_{\mathbb B}$ and a stationary martingale differences sequence $(d_n)_{n \geq 1} \in \mathbb L^2_{\mathbb B}$ such that 
    \begin{align*}
        X_n = d_n + z_{n -1} - z_n 
    \end{align*}
    (see for instance \parencite[]{Vol1993}{}). 
    Let us denote $M_n$ the martingale associated to $(d_n)_n$: $M_n = d_1 + \cdots + d_n$. 
    Then, according to \parencite[Proposition 3.2]{Cun2017}{}, condition \ref{hyp:CLT} is verified. It remains to prove that $\bigp{n^{-1}\normB{S_n}^2}_n$ is uniformly integrable. 
    First, note by stationarity of $(z_n)_n$ that 
    \begin{align*}
        \E\bigp{\frac{\normB{S_n - M_n}^2}{n}} = \frac{\E(\normB{z_0 - z_n}^2)}{n} \leq 4\frac{\E (\normB{z_0}^2)}{n} \xrightarrow[n \to \infty]{} 0. 
    \end{align*}
    Thus, $\bigp{\frac{\normB{S_n - M_n}^2}{n}}_{n \geq 1}$ is uniformly integrable as a bounded family and it is sufficient to prove that $\bigp{n^{-1}\normB{M_n}^2}_{n \geq 1}$ is uniformly integrable. Let $B >0$ and $M_n = M_n' + M_n''$ where 
    \begin{align*}
        & M_n' = \sum_{k = 1}^n d_k', \quad d_k' = d_k \indic_{\normB{d_k}\leq B} - \condEp{d_k \indic_{\normB{d_k}\leq B}}{\calF_{k -1}},
        \\
        & M_n'' = \sum_{k = 1}^n d_k'', \quad d_k'' = d_k \indic_{\normB{d_k}> B} - \condEp{d_k \indic_{\normB{d_k}> B}}{\calF_{k -1}}. 
    \end{align*}
    In the one hand, from (\ref{def:2-smooth_prp}), 
    \begin{align*}
        \frac{1}{n}\E\bigp{\normB{M_n''}^2 \indic_{\normB{M_n''}>A\sqrt n}}
        \leq 2\frac{L^2}{n} \sum_{k = 1}^n \E(\normB{d_k''}^2)
        \leq 4L^2 \E(\normB{d_0}^2 \indic_{\normB{d_0} > B}),
    \end{align*}
    which converges to 0 as $B$ tends to infinity since $d_0 \in \mathbb L^2_{\mathbb B}$. 
    \\
    On another hand, for any $a > 0$, 
    \begin{align}\label{control_Mn'}
        \frac{1}{n}\E\bigp{\normB{M_n'}^2 \indic_{\normB{M_n'}>A\sqrt n}} \leq A^{-a}n^{-(2+a)/2} \E(\normB{M_n'}^{2+a}). 
    \end{align}
    Applying \parencite[Theorem 2.6]{Pin1994}{} with $g : x \mapsto x^{2+a}$ and the martingale defined by $\Tilde{M_i'} = M_i'$ if $i \leq n$ and $\Tilde{M_i'} = M_n'$ if $i > n$, 
    \begin{align}\label{Burkholder_Pin}
        \E(\normB{M_n'}^{2+a}) 
        \leq c \E\bigp{\sum_{k = 1}^n \normB{d_k'}^2}^{\frac{2+a}{2}}
        \leq c (2B)^{2+a} n^{\frac{2+a}{2}}. 
    \end{align}
    Combining (\ref{control_Mn'}) and (\ref{Burkholder_Pin}), we infer that 
    \begin{align*}
        \frac{1}{n}\E\bigp{\normB{M_n'}^2 \indic_{\normB{M_n'}>A\sqrt n}} \leq c(2B)^{2+a}A^{-a} \xrightarrow[A \to \infty]{} 0. 
    \end{align*}
    Finally, 
    \begin{align*}
        \lim_{A \to \infty} \limsup_{n \to \infty} \frac{1}{n}\E\bigp{\normB{M_n}^2 \indic_{\normB{M_n}>A\sqrt n}} = 0, 
    \end{align*}
    that is $\bigp{n^{-1}\normB{M_n}^2}_{n \geq 1}$ is uniformly integrable, so that condition \ref{hyp:u.i.} holds. 
    \\
    So, overall, applying Theorem \ref{theoreme:vitesse} and considering Lemma \ref{prp:mixing_tau}, Proposition \ref{prp:2smooth+tau} follows. 
    
\subsection{Proof of Lemma \ref{lem:Young_tau_estimate}}
    Let $f \in \Lambda_1(\R^2, \abs{\cdot}^{1/p})$, setting $\displaystyle h_l(x) = K^l(f(x, \cdot))(x)$, we have 
    \begin{align*}
                \condEc{f(Z_k, Z_{k+l})}{Z_0}
                - 
                \E \bigcro{f(Z_k, Z_{k+l})}
        =
                K^k({h_l})(Z_0) - \nu(K^k({h_l}))
                .
    \end{align*}
    Let $\varphi$ be a bounded measurable function. One has 
    \begin{align*}
        \nu (\varphi K^k h_l) 
        &=
        \int \varphi \circ T^k(x) K^l(f(x,\cdot))(x) \,\nu(dx)
        \\&=
        \int \varphi \circ T^{k+l}(z) f(T^l(z),z) \, \nu(dz)
        \\&= 
        \int \varphi \circ \pi \circ F^{k+l} (x) f(\pi \circ F^l(x), \pi(x)) \, \bar\nu(dx)
        \\& = 
        \int \varphi \circ \pi (z) P^k\Tilde{h_l}(z) \, \bar\nu(dz)
        \\&= 
        \int \varphi(z) \E_{\bar \nu}(P^k \Tilde{h_l} | \pi = z) \, \bar\nu(dz)
    \end{align*}
    where 
    $\Tilde{h_l}(x) = P^l(f(\pi(x), \pi(\cdot)))(x)$.
    Therefore, 
    $
        K^k h_l \circ \pi
        = 
        {\bar\nu}(P^k \Tilde{h_l} | \pi)
    $
    so that
    \begin{align*}
        \sup_{f \in \Lambda_1(\R^2, \abs{\cdot}^{1/p})} \abs{K^k(h_l)(x) - \nu(K^k h_l)} 
        \leq 
        \E_{\bar \nu}\bigp{\left.\sup_{f \in \Lambda_1(\R^2, \abs{\cdot}^{1/p})}
            \abs{
                P^k(\Tilde{h_l})(Y_0) - \bar\nu(P^k(\Tilde{h_l}))
            } \right| \pi = x}. 
    \end{align*}
    It follows that
    \begin{align}
            \label{Lp_intermittent:Lemma2.1_DP09_adapted}
        2\tau_{\abs{\cdot}^{1/p}}(\calF_0, (Z_k, Z_{k+l}))
        \leq
        \E_{\bar \nu}\bigp{
            \sup_{f \in \Lambda_1(\R^2, \abs{\cdot}^{1/p})}
            \abs{
                P^k(\Tilde{h_l})(Y_0) - \bar\nu(P^k(\Tilde{h_l}))
            }
        }
    \end{align}
    where $(Y_n)_{n \in \Z}$ is a Markov chain with stationary law $\bar \nu$ and Kernel operator $P$. 
    Now, for any $x,y \in X$, 
    \begin{align}
        \abs{\Tilde{h_l}(x) - \Tilde{h_l}(y)}
        \leq 
            \nonumber
        &
        \abs{P^l f(\pi(x), \pi(\cdot))(x) - P^l f(\pi(y), \pi(\cdot))(x)}
        \\&
        +
        \abs{P^l f(\pi(y), \pi(\cdot))(x) - P^l f(\pi(y), \pi(\cdot))(y)}.
            \label{Lp_intermittent:lip_hl}
    \end{align}
    For any $f \in \Lambda_1(\R^2, \abs{\cdot}^{1/p})$, since $\pi$ is Lipschitz with respect to $d$, 
    \begin{align}
        \abs{P^l f(\pi(x), \pi(\cdot))(x) - P^l f(\pi(y), \pi(\cdot))(x)}
        &\leq 
        \kappa ^{1/p} d^{1/p}(x,y). 
            \label{Lp_intermittent:lip_hl_terme1}
    \end{align}
    On the other hand, 
    \begin{align*}
        \abs{P^l f(\pi(y), \pi(\cdot))(x) - P^l f(\pi(y), \pi(\cdot))(y)}
        &=
        \abs{P^l\psi_y(x) - P^l\psi_y(y)}
    \end{align*}
    with $\psi_y(\cdot) = f(\pi(y), \pi(\cdot))$. But if $f \in \Lambda_1(\R^2, \abs{\cdot}^{1/p})$, $\psi_y$ is Lipschitz with respect to the distance $d^{1/p}$. Hence, applying Lemma 2.2 in \cite{DP2009} it follows that there exists $c>0$ such that for any $f \in \Lambda_1(\R^2, \abs{\cdot}^{1/p})$,
    \begin{align}
        \abs{P^l f(\pi(y), \pi(\cdot))(x) - P^l f(\pi(y), \pi(\cdot))(y)}
        &\leq (c\kappa)^{1/p}d^{1/p}(x,y).
            \label{Lp_intermittent:lip_hl_terme2}
    \end{align}
    Finally combining (\ref{Lp_intermittent:lip_hl}), (\ref{Lp_intermittent:lip_hl_terme1}) and (\ref{Lp_intermittent:lip_hl_terme2}), $\Tilde{h_l}$ is Lipschitz with respect to the distance $d^{1/p}$ with Lipschitz constant $\kappa^{1/p} + (c\kappa)^{1/p}$. 
    Thus, from Lemma 2.2 in \cite{DP2009} applied to $\Tilde{h_l}$, there exists $c > 0$ such that
    \begin{align}
            \label{Lp_intermittent:lemma2.2}
        &\bar \nu \bigp{\sup \left\{ \abs{P^k \Tilde{h_l}(Z_0) - \bar \nu (P^k\Tilde{h_l})} : f \in \Lambda_1(\R^2, \abs{\cdot}^{1/p} \right\} }
        \\
        &\hspace{150pt}\leq
        c\bar \nu \bigp{\sup \left\{ \abs{P^k \varphi(Z_0) - \bar \nu (\varphi)} : \varphi \in \Lambda_1(X, d^{1/p}) \right\} }
        .
            \nonumber
    \end{align}
    Consequently, combining (\ref{Lp_intermittent:Lemma2.1_DP09_adapted})  and (\ref{Lp_intermittent:lemma2.2}) with estimate (4.3) in \cite{DM2015}, there exists $c > 0$ independent of $k$ and $l$ such that
    \begin{align*}
        \tau_{\abs{\cdot}^{1/p}}(\calF_0, (Z_k, Z_{k+l})) \leq ck^{-(1-\gamma)/\gamma}. 
    \end{align*}

\appendix
\section{Fréchet-derivative of order 3 of \texorpdfstring{$\norm{\cdot}_{L^p(\mu)}^q$}{a}}\label{section:annex}
In this section, we prove Lemma \ref{lemma:diff3_normLp}. For the sake of clarity, in what follows let us use the notation $\norm{\cdot}_p$ to refer to $\norm{\cdot}_{L^p(\mu)}$.

\noindent 
Define the function $\ell$ from $L^p(\mu)$ to $\R$ by 
\begin{align*}
    \ell(x) = \norm{x}_p^q. 
\end{align*}
\\
By combining a second order Taylor's integral formula and Hölder's inequality, we first note that for any $x,u \in L^p(\mu)\hspace{-3pt}\setminus\hspace{-3pt}\{0\}$, 
\begin{align*}
    \norm{x+u}_p^p - \norm{x}_p^p 
    = p \int u.sgn(x).\abs x ^{p-1} \, d\mu + O(\norm{u}_p^2). 
\end{align*}
Consequently, applying a first order Taylor's formula, 
\begin{align}\label{normp_diff:l_DL_1}
    \ell(x+u) - \ell(x) 
    = (\norm{x+u}_p^p)^{q/p} - (\norm{x}_p^p )^{q/p}
    = q\ell(x)^{1-p/q} \int u. x\abs x ^{p-2} \, d\mu + o(\norm{u}_p). 
\end{align}
Moreover, $
    \ell(u) - \ell(0) = o(\norm{u}_p)
    $
if $q > 1$.  
Therefore, as soon as $p \geq 1$, $\ell$ is Fréchet-differentiable on $L^p(\mu)$ with 
\begin{align}
    \ell^{(1)}(0)(u) = 0 
    \quad \text{and} \quad
    \ell^{(1)}(x)(u) = q\ell(x)^{1-p/q} \int u. x\abs x ^{p-2} \, d\mu \text{ if } x\neq 0.
        \label{normp_diff:l_diff1}
\end{align}

Let us prove that $\ell$ is twice Fréchet-differentiable. Write, for any $x, u, v \in L^p(\mu)$ with $x \neq 0, v \neq 0$,
\begin{align}\label{normp_diff:l_DL_2_long}
    \hspace{44pt}&\hspace{-44pt}
    \ell^{(1)}(x+v)(u) - \ell^{(1)}(x)(u) \nonumber
    \\
    =& q\ell(x+v)^{1-p/q} \int u. (x+v)\abs{x+v} ^{p-2} \, d\mu - q\ell(x)^{1-p/q} \int u. x\abs x ^{p-2} \, d\mu \nonumber
    \\
    =& 
        q\bigcro{\ell(x+v)^{1-p/q} - \ell(x)^{1-p/q}} \int u.\bigcro{(x+v)\abs{x+v} ^{p-2} - x\abs{x} ^{p-2}} \, d\mu \nonumber
        \\&+ 
        q\ell(x)^{1-p/q} \int u.\bigcro{(x+v)\abs{x+v} ^{p-2} - x\abs{x} ^{p-2}} \, d\mu \nonumber
        \\&+ 
        q\bigcro{\ell(x+v)^{1-p/q} - \ell(x)^{1-p/q}}\int u.x.\abs x ^{p-2} \, d\mu. 
\end{align}
On the one hand, combining a second order Taylor's integral formula with (\ref{normp_diff:l_DL_1}), 
\begin{align}\label{normp_diff:l(1-p/3)_DL_1}
    \ell(x+v)^{1-p/q} - \ell(x)^{1-p/q}
        \nonumber
    =& 
    \bigp{1 - \frac{p}{q}}\ell(x)^{-p/q}\bigcro{\ell(x+v) - \ell(x)} 
    \\
    &+
    \bigp{1 - \frac{p}{q}}\bigp{- \frac{p}{q}}\bigcro{\ell(x+v) - \ell(x)}^2 \int \bigcro{\ell(x) +t(\ell(x+v) - \ell(x))}^{-p/q-1} \, dt
        \nonumber
    \\
    =& (q-p)\ell(x)^{1-2p/q} \int vx\abs x^{p-2} \, d\mu + o(\norm{v}_p). 
\end{align}
On the other hand, on $\R$, 
\begin{align}\label{normp_diff:x(p-1)_DL_1}
    (x+v)\abs{x+v} ^{p-2} - x\abs{x} ^{p-2}
    & 
    = (p-1)v\abs x^{p-2} + o(\abs{v}). 
\end{align}
Hence, combining (\ref{normp_diff:l_DL_2_long}), (\ref{normp_diff:l(1-p/3)_DL_1}) and (\ref{normp_diff:x(p-1)_DL_1}), 
\begin{align*}
    \ell^{(1)}(x+v)(u) - \ell^{(1)}(x)(u)
    =& q(p-1)\ell(x)^{1-p/q} \int uv\abs x^{p-2} \, d\mu
    \\
    + &
    q(q-p)\ell(x)^{1-2p/q}\int vx\abs x^{p-2} \, d\mu \int ux\abs x^{p-2} \, d\mu 
    + o(\norm{v}_p). 
        \nonumber
\end{align*}
Moreover, $
    \ell^{(1)}(v)(u) - \ell^{(1)}(0)(u) = o(\norm{v}_p)
    $
if $q >2$. 
We deduce that provided that $p \geq 2$, $\ell$ is twice Fréchet-differentiable on $L^p(\mu)$ with $\ell^{(2)}(0)(u,v) = 0$ and for $x \neq 0$,
\begin{align}
    \ell^{(2)}(x)(u,v) 
    = q(p-1)\ell(x)^{1-p/q} \int uv\abs x^{p-2} \, d\mu
    + 
    q(q-p)\ell(x)^{1-2p/q}\int vx\abs x^{p-2} \, d\mu \int ux\abs x^{p-2} d\mu. 
        \label{normp_diff:l_diff2}
\end{align}

Finally, let us prove that $\ell$ is thrice Fréchet-differentiable on $L^p(\mu)$. For any $x, u, v, w \in L^p(\mu)$, with $x \neq 0, w \neq 0$, 
\begin{align}\label{normp_diff:l_DL_3_long}
    \hspace{17pt}&\hspace{-17pt}
    \ell^{(2)}(x + w)(u,v) - \ell^{(2)}(x)(u,v)
    \\
    &\hspace{-12pt}= 
    q(p-1)\ell(x+w)^{1 - p/q}\int uv \abs{x + w}^{p-2}d\mu
    - 
    q(p-1)\ell(x)^{1 - p/q}\int uv \abs{x}^{p-2}d\mu
        \nonumber
    \\
    &
    + 
    q(q-p) \ell(x + w)^{1 - 2p/q}\int v(x + w) \abs{x + w}^{p-2}d\mu \int u (x+w)\abs{x + w}^{p-2}d\mu 
        \nonumber
    \\
    &
    -
    q(q-p) \ell(x)^{1 - 2p/q}\int vx \abs{x}^{p-2}d\mu \int u x\abs{x}^{p-2}d\mu 
        \nonumber
    \\
    &
    \hspace{-12pt}=
    q(p-1)\ell(x+w)^{1-p/q}\int uv\bigcro{\abs{x+w}^{p-2} - \abs{x}^{p-2}} d\mu 
        \nonumber
    \\
    &+ 
    q(p-1)\bigcro{\ell(x+w)^{1-p/q}-\ell(x)^{1-p/q}}\int uv \abs{x}^{p-2} d\mu 
        \nonumber
    \\
    &+
    q(q-p)\ell(x+w)^{1-2p/q}\int v\bigcro{(x+w)\abs{x+w}^{p-2} - x\abs{x}^{p-2}}d\mu \int u \bigcro{(x+w)\abs{x+w}^{p-2} - x\abs{x}^{p-2}} d\mu
        \nonumber
    \\
    &+
    q(q-p)\ell(x+w)^{1-2p/q}\int v\bigcro{(x+w)\abs{x+w}^{p-2} - x\abs{x}^{p-2}}d\mu \int u x\abs{x}^{p-2} d\mu
        \nonumber
    \\
    &+
    q(q-p)\ell(x+w)^{1-2p/q} \int vx\abs{x}^{p-2} d\mu \int u\bigcro{(x+w)\abs{x+w}^{p-2} - x\abs{x}^{p-2}} d\mu 
        \nonumber
    \\
    &+
    q(q-p)\bigcro{\ell(x+w)^{1-2p/q}-\ell(x)^{1-2p/q}} \int vx\abs{x}^{p-2} d\mu \int ux\abs x^{p-2} d\mu
    . 
        \nonumber
\end{align}
Note that
\begin{align}\label{normp_diff:l(1-2p/3)_DL_1}
    \ell(x+w)^{1-2p/q} - \ell(x)^{1-2p/q}
    = 
    (q-2p)\ell(x)^{1-3p/q} \int wx\abs x^{p-2} \, d\mu + o(\norm{w}_p) 
\end{align}
and on $\R$, 
\begin{align}\label{normp_diff:x(p-2)_DL_1}
    \abs{x+w} ^{p-2} - \abs{x} ^{p-2}
    & 
    = (p-2) w x \abs x^{p-4} + o(\abs{w}). 
\end{align}
Applying (\ref{normp_diff:l(1-p/3)_DL_1}), (\ref{normp_diff:x(p-1)_DL_1}), (\ref{normp_diff:l(1-2p/3)_DL_1}) and (\ref{normp_diff:x(p-2)_DL_1}) in (\ref{normp_diff:l_DL_3_long}) infers that
\begin{align}\label{normp_diff:l_DL2}
    \hspace{44pt}&\hspace{-44pt}\ell^{(2)}(x+w)(u,v) - \ell^{(2)}(x)(u,v) 
    \\
    =& \;
    q(p-1)(p-2)\ell(x)^{1-p/q} \int u v w x\abs{x}^{p-4} d\mu 
        \nonumber
    \\&
    +
    q(p-1)(q-p)\ell(x)^{1-2p/q}\int wx\abs{x}^{p-2}d\mu \int uv \abs{x}^{p-2}d\mu
        \nonumber
    \\&
    +
    q(p-1)(q-p)\ell(x)^{1-2p/q}\int vx\abs{x}^{p-2}d\mu \int uw \abs{x}^{p-2}d\mu
        \nonumber
    \\&
    +
    q(p-1)(q-p)\ell(x)^{1-2p/q}\int ux\abs{x}^{p-2}d\mu \int vw \abs{x}^{p-2}d\mu
        \nonumber
    \\&
    +
    q(q-p)(q-2p)\ell(x)^{1-3p/q} \int wx\abs{x}^{p-2}d\mu \int vx\abs{x}^{p-2}d\mu \int ux\abs{x}^{p-2}d\mu
        \nonumber
    \\&
    +
    o(\norm{w}_p). 
        \nonumber
\end{align}
Therefore, as soon as $p \geq 3$, $\ell$ is thrice Fréchet-differentiable on $L^p(\mu) \setminus \{0\}$ and for $x \neq 0$
\begin{align}
    \hspace{33pt}&\hspace{-33pt}\ell^{(3)}(x)(u,v,w) 
    =
    q(p-1)(p-2)\norm{x}_p^{q - p} \int uvwx\abs{x}^{p-4} d\mu 
        \nonumber
    \\&
    +
    q(p-1)(q-p)\norm{x}_p^{q - 2p}\left(\int wx\abs{x}^{p-2}d\mu \int uv \abs{x}^{p-2}d\mu\right.
    +
    \int vx\abs{x}^{p-2}d\mu \int uw \abs{x}^{p-2}d\mu
        \nonumber
    \\&
    \hspace{10cm}
    +
    \left.\int ux\abs{x}^{p-2}d\mu \int vw \abs{x}^{p-2}d\mu \right)
        \nonumber
    \\&
    +
    q(q-p)(q-2p)\norm{x}_p^{q - 3p} \int wx\abs{x}^{p-2}d\mu \int vx\abs{x}^{p-2}d\mu \int ux\abs{x}^{p-2}d\mu. 
\end{align}
~\\
\noindent\textbf{Acknowledgements. }I would like to thank my two advisors J\'er\^{o}me Dedecker and Florence Merlevède for helpful discussions. 
\printbibliography
\end{document}